\documentclass[12pt,a4paper]{amsart}
\usepackage{amsfonts,amsmath,amssymb,graphicx,amscd,amsthm}
\usepackage{lmodern}
\usepackage[T1]{fontenc}
\usepackage[all,cmtip]{xy}
\usepackage{hyperref}
\newtheorem{theorem}{Theorem}[section]
\newtheorem{proposition}[theorem]{Proposition}
\newtheorem{corollary}[theorem]{Corollary}
\newtheorem{lemma}[theorem]{Lemma}
\theoremstyle{definition}
\newtheorem{definition}[theorem]{Definition}
\newtheorem{definitions}[theorem]{Definitions}
\newtheorem{properties}[theorem]{Properties}
\newtheorem{remark}[theorem]{Remark}
\newtheorem{example}[theorem]{Example}
\newtheorem{examples}[theorem]{Examples}
\newtheorem{notations}[theorem]{Notations}
\newtheorem{recalls}[theorem]{Recalls}
\newcommand{\bull}{\subsection}
\newcommand{\modcat}[1]{{#1}\mathrm{-mod}}
\newcommand{\cmod}{\modcat{\mathcal C}}
\newcommand{\copmod}{\modcat{\mathcal C^{\text{op}}}}
\newcommand{\Hom}{\textup{Hom}}
\newcommand{\Ch}{\textup{Ch}}
\newcommand{\Ab}{\mathbf{Ab}}
\author[Ramzi Ksouri]{Ramzi Ksouri$^{**}$}
\title[Duality categories]{Duality categories$^*$}
\address{Laboratoire de Math\'ematiques: Alg\`ebre, Topologie, Arithm\'etique et Structures ordonn\'ees "LATAO", D\'epartement de Math\'ematiques, Facult\'e des Sciences de Tunis, Le Campus Universitaire TN-2092, Tunis, Tunisie.
}
\email{ksouri.r@gmail.com}
\thanks{$^*$ The final publication is available at Springer via \url{http://dx.doi.org/10.1007/s10485-015-9397-8}.}
\thanks{$^{**}$ Laboratoire de Math\'ematiques: Alg\`ebre, Topologie, Arithm\'etique et Structures ordonn\'ees "LATAO", D\'epartement de Math\'ematiques, Facult\'e des Sciences de Tunis, Le Campus Universitaire TN-2092, Tunis, Tunisie.
}
\begin{document}\sloppy

\maketitle
\date{}
\begin{abstract}
We define the notion of duality categories as generalization of duality groups. Two examples are treated. The first is the Serre duality in the categories of strict polynomial functors. The second concerns finite complexes. We show in particular that finite Tits buildings are duality categories.
\end{abstract}

\setcounter{tocdepth}{1}
\tableofcontents

\section{Introduction}
A group $G$ is a duality group in the sense of Beiri and Eckmann~\cite{BE} 
if there exist an integer $n$, a right $G$-module $D$ and a natural isomorphism
$$H^k(G,-)\cong H_{n-k}(G,D\otimes -).$$
We generalize this definition of duality group by replacing the group $G$ with a category $\mathcal{C}$.
A $\mathcal{C}$-module is then a functor from $\mathcal{C}$ to the category of abelian groups $\Ab$.
\begin{definition}
A small category $\mathcal{C}$ is said a duality category if there exists a functor $D:\mathcal{C}^{op}\rightarrow \Ab$, an integer $n$ and a natural isomorphism
$$\operatorname{Ext}_{\cmod}^i\left (\underline{\mathbb{Z}},-\right )\cong \operatorname{Tor}^{\cmod}_{n-i}\left (D,-\right )$$
where $\underline{\mathbb{Z}}$ is the constant functor.
\end{definition}
In this article, we outline 
a theory of duality. 
In the case where the small category $\mathcal{C}$ is a finite simplicial complex, we obtain the 
following local condition:
\begin{theorem}
Let $\mathcal{P}$ the partial ordered set of simplices of a finite simplicial complex. 
If there exists an integer $n$ such that for any simplex $x$ of $\mathcal{P}$ 
$$\widetilde{ \operatorname{H}}^{i}(\operatorname{link}_{x})=0 \text{   for   } i\neq n-\dim(x)-1$$
(where $\widetilde{ \operatorname{H}}^{*}$ denotes the reduced cohomology with integer coefficients)
then $\mathcal{P}$ is a duality category.
\end{theorem}
This theorem allows us to
recover the
classical Poincar\'e duality for triangulated manifolds. It also leads to the following result.
\begin{corollary}
Spherical Tits buildings are duality categories.
\end{corollary}
More generally, if $k$ denotes a commutative unitary ring and $\modcat{k}$ the category of left $k$-modules,
we say that a functor $F:\mathcal{C} \rightarrow \modcat{k}$ is a duality functor 
if there exist an integer $n$, a functor $D^n(F):\mathcal{C}^{op}\rightarrow \modcat{k}$ and a natural isomorphism
\begin{equation}\label{I1}
\operatorname{Ext}_{\cmod}^i\left (F,-\right )\cong \operatorname{Tor}^{\cmod}_{n-i}\left (D^n\left (F\right ),-\right ).
\end{equation}

If $k$ is a field, the isomorphism~(\ref{I1}) is equivalent to the following
$$\operatorname{Ext}_{\cmod}^i\left (F,-\right )'\cong \operatorname{Ext}_{\cmod}^{n-i}\left (-,\text{$D^n\left (F\right )'$}\right )$$
where $D^n(F)'$ denotes the postcomposition of $D^n(F)$ by the duality operation $V\mapsto V':=\Hom_k(V,k)$.
The latter isomorphism  is reminiscent of the classical Serre duality for projective varieties.

Idun Reiten and Michel Van Den Bergh in~\cite{RV} generalized the concept of Serre duality in the following way. 
A $k$-linear category $\mathcal{A}$ satisfies a Serre duality if it has a Serre functor,
\textit{i.e.} an endofunctor
$S:\mathcal{A} \rightarrow\mathcal{A}$ 
such that we have a natural isomorphism
$$Hom_{\mathcal{A}}(a,b)'\cong Hom_{\mathcal{A}}(b,Sa).$$
The duality presented in this paper makes explicit a Serre duality for the category $\mathcal{D}_{\mathrm{perf}}(\modcat{\mathcal{C}})$ of perfect complexes of $\modcat{\mathcal{C}}$.
Let us be clear what this is about.
Let $D$ the functor:
$$\begin{array}{crcl}
D:&(\modcat{\mathcal{C}})^{op}&\rightarrow &\modcat{\mathcal{C}^{op}}\\
& F&\mapsto& DF
\end{array}$$
defined for an object $x$ of $\mathcal{C}$ by
$$DF(x)=\Hom_{\modcat{\mathcal{C}}}(F, P_x^{\mathcal{C}})$$
where $P_x^{\mathcal{C}}=k[\Hom_{\mathcal{C}}(x,-)]$ is a standard projective of $\modcat{\mathcal{C}}$ ( the notation $k[S]$ designates the free $k$-module generated by $S$). 
Let
$$RD:\mathcal{D}^b(\modcat{\mathcal{C}})^{op} \rightarrow \mathcal{D}^b(\modcat{\mathcal{C}^{op}})$$ 
the derived functor of $D$. We prove the following result:
\begin{proposition}
The functor $RD(-)'$ is a Serre Duality for $\mathcal{D}_{\mathrm{perf}}(\mathcal{C}\mathrm{-mod})$.
\end{proposition}
We illustrate the concept of duality functor with an example in the category $\operatorname{Rep}\Gamma^d_k$ of strict polynomial functors
over a finite field $ k $ of characteristic $ p $.
Let $s,r$ be integers and $\mathfrak{s}=\left (s_1,\cdots, s_k\right )$ a partition of $s$. 
The strict polynomial functor $\Gamma^{\mathfrak{s}(r)}$ (resp. $S^{\mathfrak{s}(r)}$) of degree $d:=sp^r$
is the precomposition of $\Gamma^{\mathfrak{s}}$ (resp. $S^{\mathfrak{s}}$) by the Frobenius twist $\operatorname{I}^{(r)}$. 
Recall that $\Gamma^{\mathfrak{s}}$ (resp. $S^{\mathfrak{s}}$) is the tensor product $\Gamma^{s_1}\otimes \cdots \otimes \Gamma^{s_k}$ (resp. $S^{s_1}\otimes \cdots \otimes S^{s_k}$)  where $\Gamma$ (resp. $S$) is the functor of divided (resp. symmetric) powers. We use the properties of the Ringel duality studied in~\cite{Touze1} to show the following result:
\begin{proposition}The strict polynomial functor $\Gamma^{\mathfrak{s}(r)}$ is a duality functor and we have a natural isomorphism
$$\operatorname{Ext}_{\operatorname{Rep}\Gamma^d_k}^i\left (\Gamma^{\mathfrak{s}(r)},-\right )'\cong \operatorname{Ext}_{\operatorname{Rep}\Gamma^d_k}^{2s\left (p^r-1\right )-i}\left (-,S^{\mathfrak{s}(r)} \right).$$
\end{proposition}
 \medskip
Here is a summary of the different sections.
\par\medskip
Section $2$: This section is a collection of reminders and notations used in the article.
\par\medskip
Section $3$: In the case of modules over a ring $R$, the structure of right $R$-module on the dual 
$M'=\Hom_R(M,R)$ of a left $R$-module $M$ is inherited from the structure of $R$-bimodule of $R$. 
An equivalent of the bimodule $R$ is the bifunctor 
$$\begin{array}{crcl}\underline{\underline{\mathcal{C}}}:&\mathcal{C}^{op}\times\mathcal{C}&\rightarrow &k\mathrm{-mod}\\
& (x,y)&\mapsto& k[\Hom_{\mathcal{C}}(x,y)]\end{array}$$ 
Our duality theory is based on the following property
\begin{equation}\label{ast}
\Hom_{\cmod}(\underline{\underline{\mathcal{C}}}(y,-),\underline{\underline{\mathcal{C}}}(x,-)) \cong \underline{\underline{\mathcal{C}}}(x,y)
\tag{$\ast$}\end{equation}
provided by the Yoneda lemma. 
We call a bifunctor  with the property~\eqref{ast} a Yoneda bifunctor. We show that there is a $1-1$ correspondance between equivalence classes of epireflective and coreflective subcategories of $ \mathcal{C} \mathrm{-mod} $ and isomorphism classes of Yoneda bifunctors for $ \mathcal{C} $.
\par\medskip
Section $4$: 
We show that a Yoneda bifunctor $T$ for a category $\mathcal{C}$ provides a functor of duality $D$ 
and we study its first properties. 
We illustrate this point of view by showing (proposition~\ref{dualitedeserre}) that if $T$ is  Yoneda bifunctor for a $k$-linear category $\mathcal{C}$ then
the functor $RD(-)'$ is a Serre duality of $\mathcal{D}_{\mathrm{perf}}(\mathcal{C}_T\mathrm{-mod})$.
\par\medskip
Section $5$: A functor $F$ of $\cmod$ is a duality functor
if there exists an integer $n$ such that one has an isomorphism 
$$RD\left (F\right )\cong R^nD\left (F\right )[n]$$
in $\mathcal{D}^+\left (\mathcal{C}\mathrm{-mod}\right )$. In this section, we give several equivalent conditions to the latter. 
\par\medskip
Section $6$:
We illustrate the notions of duality presented in this article by the category $\operatorname{Rep}\Gamma^d_k$ of strict polynomial functors of degree $d$
over a finite field $k$.
We equip this category with a Serre duality functor $RD(-)'$
and we prove that the strict polynomial functor $\Gamma^{\mathfrak{s}(r)}$ is a duality functor.
 \par\medskip
Section $7$: The rest of this article is devoted to the notion of duality category. A category is a duality category if the constant functor $\underline{\mathbb Z}$ is a duality functor. 
We isolate a sufficient local criterion for a category to be a duality category.
The theorem~\ref{th1} presents this criterion in the case where the category is the partially ordered set of simplices of a simplicial complex.
We thus recover the classical Poincar\'e duality for manifolds.
It also follows that spherical Tits buildings are duality categories.
\subsection*{Acknowledgments.}
I thank Ga\"el Collinet who drawed my
attention on the
problem of generalizing
the theory of duality
groups and supplied a
constant help until the writing of this article. I also
wish to thank Antoine Touz\'e for discussions on Ringel duality.
\section{Preliminaries}
In this section, the notation $\mathcal{C}$ designates an essentially small category,
the notation $k$ designates a unitary commutative ring and 
the notation $\modcat{k}$
designates the category of left $k$-modules. 
The category $\modcat{\mathbb Z}$ is denoted $\Ab$.
\bull{The category  $\cmod$.} 
\begin{definitions}\ 
\begin{itemize}
\item A left (resp. right) $\mathcal{C}$-module is a  covariant functor from  $\mathcal{C}$ (resp. $\mathcal{C}^{op}$) to the category $\modcat{k}$.
\item We denote $I$ the  category $0\rightarrow 1$.
Let $F$ and $G$ be two left $\mathcal{C}$-modules. A natural transformation between $F$ and $G$ is a functor $$T:\mathcal{C}\times I\rightarrow k\mathrm{-mod}$$ such that $T(c,0)=F(c)$ and $T(c,1)=G(c)$. 
\item We denote $\cmod$ (resp. $\copmod$) 
the category whose objects are left (resp. right) $\mathcal{C}$-modules and whose morphisms are natural transformations.
\end{itemize}
\end{definitions}
\begin{properties}\label{propreties}\
\begin{itemize}
\item For two objects $F$ and $G$ of $\cmod$, the set
  $\Hom_{\cmod}(F,G)$ of natural transformations between $F$ and $G$
  is naturally provided with a $k$-module structure.
\item The bifunctor 
$$\Hom_{\cmod}(-,-): (\cmod)^{\text{op}} \times \cmod \rightarrow \modcat{k}$$
commutes with limits with respect to each variable 
(limit objects in $(\cmod)^{op}$ identify with colimit objects in $\cmod$). 
\item Let $P_a^{\mathcal{C}}$ denotes the functor $k[\Hom_{\mathcal{C}}(a,-)]$. The linear analogue of the Yoneda Lemma (\protect{\cite[III.2]{categories}}) provides the following natural isomorphism
\begin{equation}\label{yonneda}\Hom_{\cmod}(P_a^{\mathcal{C}},F)\cong F(a).\end{equation}
The isomorphism is given by the following map 
$$\begin{array}{rcl}\Hom_{\cmod}(P_a^{\mathcal{C}},F)&\rightarrow &F(a)\\
N&\mapsto& N_a(\mathbf{1}_a)\end{array}.$$
The inverse is defined by associating to $\lambda \in F(a)$ the natural transformation $N:P_a^{\mathcal{C}}\rightarrow F$ whose component $N_x$ is given by $N_x(f)=F(f)\lambda$ for $f\in \Hom_{\mathcal{C}}(a,x)$. 

\item The categories $\cmod$ and $\copmod$ are abelian. This follows from the definition of an abelian category and the fact that kernels, cokernels and more generally limits and colimits are calculated pointwise (see \protect{\cite[1.6.4 and A.4]{weibel}}).
\item For any object $a$ of $\mathcal{C}$, the functors called standard projectives
$P_a^{\mathcal{C}}:=k[\Hom_{\mathcal{C}}(a,-)]$ and
$P^a_{\mathcal{C}}:= k[\Hom_{\mathcal{C}}(-,a)]$
are projective in  $\cmod$ and $\copmod$ respectively. This is a consequence of the Yoneda isomorphism \ref{yonneda} and the fact that a short exact sequence in $\cmod$ is exact if and only if it's exact pointwise. 

\item The categories $\cmod$ and $\copmod$  have enough projectives. The standard projectives of $\cmod$ (resp. of $\copmod$) 
form a system of projective generators. This is a particular case of corollary \ref{CT}. They also have enough injectives (see \protect{\cite[theorem 1.10.1]{Grothendieck}}).

\end{itemize}
\end{properties}

\bull{Tensor product over $\mathcal{C}$.}
Let $F$ be a right $\mathcal{C}$-module and $G$ a left $\mathcal{C}$-module. The tensor product $F\otimes_{\mathcal{C}}G$ of $F$ and $G$ over $\mathcal{C}$ is the quotient of the $k$-module $$\bigoplus_{x\in \text{Ob}\mathcal{C}}F(x)\otimes_{k}G(x)$$ by the submodule generated by the elements
$$F(f)a\otimes b-a\otimes G(f)b~~~~(\forall f\in
\Hom_{\mathcal{C}}(x,y), \forall a\in F(y), \forall b\in G(x)).$$
\begin{properties}\ 
\begin{itemize}
\item The tensor product over $\mathcal{C}$ defines a bifunctor
$$-\otimes_{\mathcal{C}}-:\copmod \times \cmod \rightarrow \modcat{k}$$
which commutes with colimits with respect to each variable.
In particular,
the bifunctor $-\otimes_{\mathcal{C}}-$ is right exact relative to each variable. 
\item We have the following natural isomorphisms
\begin{equation} \label{yonnedat}P^x_{\mathcal{C}}\otimes_{\mathcal{C}}G\cong
  G(x)\text{    and    }F\otimes_{\mathcal{C}}P_x^{\mathcal{C}} \cong
  F(x).\end{equation}
\item For any $k$-module $M$, and any $F$ in $\copmod $,
we denote $\underline{\Hom}(F,M)$ the
$\mathcal{C}$-module given by $x \mapsto \Hom_{k\mathrm{-mod}}(F(x),M)$. The adjunction 
$$\Hom_{k}(F\otimes_{\mathcal{C}}G, M)\cong \Hom_{\cmod}(G,\underline{\Hom}(F,M))$$
characterizes the tensor product up to natural isomorphism.\\
\end{itemize}
\end{properties}
\bull{Category of complexes.}
Let  $\mathcal{A}$ be an abelian category. 
We denote $\Ch(\mathcal{A})$ the abelian category whose objects are cochain complexes of objects of $\mathcal{A}$  and whose morphisms are morphisms of complexes. 
the category $\Ch(\mathcal{A})$ contains the full subcategories $\Ch^b(\mathcal{A})$, $\Ch^-(\mathcal{A})$
 and $\Ch^+(\mathcal{A})$ of bounded, bounded above and bounded below cochain complexes. The notation $\Ch^*(\mathcal{A})$ designates one of these subcategories.
\par
Let $X=(X^n,d_X^n)$ and $Y=(Y^n,d_Y^n)$ be two complexes of $\Ch(\mathcal{A})$. 
The bifunctor $\Hom_{\mathcal A} (-,-):\mathcal{A}^{op}\times \mathcal{A}\rightarrow \modcat{k}$ 
induces a bifunctor
$$\Hom(-,-):\Ch(\mathcal{A})^{op}\times \Ch(\mathcal{A})\rightarrow \Ch(\modcat{k})$$
defined by
$$\left(\Hom(X,Y)\right)^n=\prod_{i\in \mathbb{Z}}\Hom_{\mathcal{A}}(X^i,Y^{i+n})$$
and $$d^nf=d_Y\circ f+(-1)^{n+1}f\circ d_X,~~f\in (\Hom(X,Y))^n.$$
Let $X=(X^n,d_X^n)$ be a complex of $\Ch(\copmod)$ and $Y=(Y^n,d_Y^n)$ a complex of $\Ch(\cmod)$.  
The bifunctor $-\otimes_{\mathcal{C}}-:\copmod\times \cmod\rightarrow \modcat{k} $ induces a bifunctor
$$-\otimes-:\Ch(\copmod)\times \Ch(\cmod)\rightarrow \Ch(\modcat{k})$$
defined by
$$\left( X\otimes
Y\right)^n=\bigoplus_{p+q=n}X^p\otimes_{\mathcal{C}}Y^q\ \ \ ,$$
the differentials being defined, for $a \in
\text{Ob}(\mathcal{C})$, $x\in X^p(a)$ and $y\in Y^q(a)$ by the formula
$$d^n(x\otimes y)=x\otimes d_Y^q y +(-1)^{q} d_X^px\otimes y\ \ \ .$$
\bull{Extensions groups over $\mathcal{C}$.}
The bifunctor 
$$\Hom_{\cmod}(-,-):\cmod^{op} \times \cmod\rightarrow \modcat{k}$$ 
commutes with limits relative to each variable, we denote 
$\text{Ext}^*_{\cmod}(-,-)$ the extensions groups over $\mathcal{C}$ defined as right derivatives of the bifunctor $\Hom_{\cmod}(-,-)$. Let $F$ and $G$ be two left $\mathcal{C}$-modules, we obtain 
$\text{Ext}^*_{\cmod}(F,G)$ 
either by choosing a projective resolution $\mathcal{P}_*$ of $F$ and computing the cohomology of the cochain complex $\text{Hom}_{\mathcal{C}\mathrm{-mod}}(\mathcal{P}_*,G)$,
or by choosing an injective resolution $\mathcal{I}^*$ of $G$ and computing the cohomology of the cochain complex $\text{Hom}_{\mathcal{C}\mathrm{-mod}}(F,\mathcal{I}^*)$.
\bull{Torsion groups over $\mathcal{C}$.}
The bifunctor 
$$-\otimes_{\mathcal{C}}-\ :\ \copmod\times\cmod \to \modcat{k}$$
is right exact with respect to each variable, we denote 
$\text{Tor}^{\cmod}_*(-,-)$ the torsion groups over $\mathcal{C}$ defined 
as the left derivatives of the tensor product over $\mathcal{C}$. For a right $\mathcal C$-module $F$ and a left $\mathcal C$-module $G$, we obtain 
$\text{Tor}^{\cmod}_*(F,G)$ either by choosing a projective resolution $\mathcal{P}_*$ of $F$ and computing the homology of the chain complex $P_*\otimes_{\mathcal{C}}G$ or by choosing a projective resolution $Q_*$ of $G$ and computing the homology of the chain complex $F\otimes_{\mathcal{C}}Q_*$.
\section{Yoneda bifunctor and epireflective subcategory}
We add the following notations to those introduced in the previous section.
\begin{notations}\label{notations}Let
  $B(-,-):\mathcal{C}^{op}\times\mathcal{C}\rightarrow k\mathrm{-mod}$ a bifunctor. Let $F$ and $G$ be
elements of $\cmod$
and $\modcat{\mathcal{C}^{op}}$
respectively. We denote 
\begin{itemize}
\item $B_a$ the left $\mathcal{C}$-module $B(a,-)$.
\item $B^a$ the right $\mathcal{C}$-module $B(-,a)$.
\item $\mathcal{H}om(F,B)$ the right $\mathcal{C}$-module $a\mapsto  \Hom_{\cmod}(F,B_a)$. 
\item $\mathcal{H}om(B,F)$ the left $\mathcal{C}$-module $a\mapsto  \Hom_{\cmod}(B_a,F)$. 
\item $B\otimes_{\mathcal{C}}F$ the left $\mathcal{C}$-module $a\mapsto B^a\otimes_{\mathcal{C}}F$.
\item $G\otimes_{\mathcal{C}}B$ the right $\mathcal{C}$-module $a\mapsto G\otimes_{\mathcal{C}}B_a$.
\end{itemize}
\end{notations}
We denote $\underline{\underline{\mathcal{C}}}$ the bifunctor $$ (x,y)\mapsto k[\Hom_{\mathcal{C}}(x,y)].$$ 
We have the equalities $\underline{\underline{\mathcal{C}}}_a=P_a^{\mathcal{C}}$ and $\underline{\underline{\mathcal{C}}}^a=P^a_{\mathcal{C}}$.
\begin{definition}
A Yoneda bifunctor for $\mathcal C$ is a bifunctor 
$T : \mathcal C^{op}\times\mathcal C \to k\mathrm{-mod}$ with a surjection
$$ \mu: \underline{\underline{\mathcal{C}}} \xymatrix{  \ar@{->>}[r]& } T$$
such that the canonical morphism between elements of $\modcat{(\mathcal C^{op}\times\mathcal C)}$
\begin{equation}\label{YBI} 
\varphi : \left[(a,b)\mapsto \Hom_{\cmod}(T_b,T_a)\right] \to \left[(a,b)\mapsto T(a,b)\right]
\end{equation}
defined by $\varphi_{(a,b)}(F)=F_b(\mu_{b,b}(\mathbf{1}_b))$ is an isomorphism where $F:T_b\rightarrow T_a$ is a natural transformation and $F_b:T_b(b)\rightarrow T_a(b)$ the component of $F$ assigned to $b$.
\end{definition}
\begin{examples}\label{yonedaexemple}~~
\begin{enumerate}
\item The bifunctor $$\begin{array}{crcl}\underline{\underline{\mathcal{C}}}:&\mathcal{C}^{op}\times\mathcal{C}&\rightarrow &k\mathrm{-mod}\\
& (x,y)&\mapsto& k[\Hom_{\mathcal{C}}(x,y)]\end{array}$$ 
with the identity
transformation as the surjection $\mu$ is a Yoneda bifunctor. This is a direct consequence of the Yoneda isomorphism (see properties \ref{propreties}).

\item Suppose $\mathcal{C}$ is a $k$-linear category. We denote $L_{\mathcal{C}}$ the bifunctor $$\begin{array}{crcl}L_{\mathcal{C}}:&\mathcal{C}^{op}\times\mathcal{C}&\rightarrow &\modcat{k}\\
& (x,y)&\mapsto& \Hom_{\mathcal{C}}(x,y)\end{array}.$$
The bifunctor $L_{\mathcal{C}}$ with the surjection 
$$\begin{array}{crcl}\mu_{x,y}:&k[\Hom_{\mathcal{C}}(x,y)]&\rightarrow &\Hom_{\mathcal{C}}(x,y)\\
& [f]&\mapsto& f\end{array}$$
is a Yoneda bifunctor. Indeed the map
$$\begin{array}{rcl}\Hom_{\cmod}\left(\Hom_{\mathcal{C}}(y,-),\Hom_{\mathcal{C}}(x,-)\right)&\rightarrow &\Hom_{\mathcal{C}}(x,y)\\
T&\mapsto& T_y(\mu_{y,y}(\mathbf{1}_y)).\end{array}$$
is an isomorphism. The inverse is given by associating to $f\in \Hom_{\mathcal{C}}(x,y)$ the natural transformation whose components are the precomposition by $f$.
\end{enumerate}
\end{examples}
\begin{proposition}\label{Y1}
Let $T$ be a Yoneda bifunctor. 
\begin{enumerate}
\item The bifunctor $T^{op}:(a,b)\mapsto T(b,a)$ is a Yoneda bifunctor for $\mathcal{C}^{op}$.
\item The canonical morphism between elements of $(\mathcal C^{op}\times\mathcal C)\mathrm{-mod}$
$$ \psi :\left[(a,b)\mapsto \Hom_{\cmod}(T_b,T_a)\right] \to \left[(a,b)\mapsto T^b\otimes_{\mathcal{C}}T_a)\right] $$
defined by $\psi_{(a,b)}(F)=\mu_{b,b}(\mathbf{1}_b)\otimes F_b(\mu_{b,b}(\mathbf{1}_b))$ is an isomorphism.
\end{enumerate}
\end{proposition}
\begin{proof}~~
\begin{enumerate}
\item The canonical morphism
$$ \varphi^{op} : \left[(a,b)\mapsto \Hom_{\mathcal C\mathrm{-mod}}(T^b,T^a)\right] \to \left[(a,b)\mapsto T(b,a)\right]$$
defined by $\varphi^{op}_{(a,b)}(F)=F_b(\mu_{b,b}(\mathbf{1}_b))$ is injective. \\
Let $\lambda \in T(b,a)$. 
The assumption on $T$ implies that there is a unique natural transformation $F:T_a\to T_b$ 
such that $F_b(\mu_{b,b}(\mathbf{1}_b))=\lambda$. 
We define a natural transformation $\tilde{F}:T^b\to T^a$ as follows: for $x$ in $\text{Ob}(\mathcal{C})$,
$$\begin{array}{crcl}
\tilde{F}_x:&T^b(x)&\rightarrow & T^a(x)\\
&\lambda'&\mapsto& \varphi_{(x,a)}\circ F^*\circ \varphi^{-1}_{(x,b)}(\lambda').
\end{array}$$
The equality $\varphi^{op}(\tilde{F})=\lambda$ 
implies the surjectivity of $\varphi^{op}$.  
\item The morphism $\psi$ of the proposition is the composition
$$\Hom_{\cmod}(T_b,T_a) \to  T(a,b) \to  P^b_{\mathcal{C}}\otimes_{\mathcal{C}}T_a \stackrel{\mu}{\rightarrow} T^b\otimes_{\mathcal{C}}T_a.$$
The first two maps are isomorphisms. 
The third is surjective since the tensor product is right exact. 
We construct an inverse $T^b\otimes_{\mathcal{C}}T_a \rightarrow \Hom_{\cmod}(T_b,T_a)$ to $\psi$. 
First, under the natural transformation (\ref{YBI}), we identify $T(x,b)$ to $\Hom_{\cmod}(T_b,T_x)$ and $T(a,x)$ to $\Hom_{\cmod}(T_x,T_a)$ for each object $x$ of $\mathcal{C}$. The composition of natural transformation gives a bilinear map 
$$\begin{array}{cccl}
&\Hom_{\cmod}(T_b,T_x)\times 
\Hom_{\cmod}(T_x,T_a) &\rightarrow &\Hom_{\cmod}(T_b,T_a)\\
&(S,R)&\mapsto& R\circ S
\end{array}$$
that induces a linear map $$\Hom_{\cmod}(T_b,T_x)\otimes_k \Hom_{\cmod}(T_x,T_a)\rightarrow \Hom_{\cmod}(T_b,T_a).$$ Taking the direct sum over objects of $\mathcal{C}$, we obtain a map 
\begin{equation}\label{YBI1}\bigoplus_{x\in \text{Ob}\mathcal{C}}{\Hom_{\cmod}(T_b,T_x)\otimes_k \Hom_{\cmod}(T_x,T_a)} \rightarrow \Hom_{\cmod}(T_b,T_a).\end{equation}
Let $f:x\rightarrow y$ a morphism in $\mathcal{C}$. Denote $T(b,f)$ (resp. $T(f,a)$) the induced morphism 
$\Hom_{\cmod}(T_b,T_y)\rightarrow \Hom_{\cmod}(T_b,T_x)$ (resp. $\Hom_{\cmod}(T_x,T_a)\rightarrow \Hom_{\cmod}(T_y,T_a)$). The equality 
$R\circ T(b,f)S=T(f,a)R\circ S$ for all $S\in \Hom_{\cmod}(T_b,T_y)$ and $R \in \Hom_{\cmod}(T_x,T_a)$ implies that the map (\ref{YBI1}) passes to quotient to give a map $T^b\otimes_{\mathcal{C}}T_a \rightarrow \Hom_{\cmod}(T_b,T_a)$ that is an inverse to $\psi$.

\end{enumerate}
\end{proof} 
\begin{notations}\label{categorieyoneda} We denote $\mathcal{C}_T\mathrm{-mod}$ the smallest full subcategory of $\cmod$ containing the $T_a$ and stable under taking limits and colimits.
This notation will be justified by corollary \ref{CT}.
\end{notations}
\begin{example}~~
\begin{enumerate}
\item $\mathcal{C}_{\underline{\underline{\mathcal{C}}}}\mathrm{-mod}=\mathcal{C}\mathrm{-mod}$.
\item $\mathcal{C}_{\underline{\underline{\mathcal{C}}}^{op}}\mathrm{-mod}=\mathcal{C}^{op}\mathrm{-mod}$.
\item If $\mathcal{C}$ is $k$-linear then $\mathcal{C}_{L_{\mathcal{C}}}\mathrm{-mod}$ is the full subcategory of $k$-linear functors of $\mathcal{C}\mathrm{-mod}$.
\end{enumerate}
\end{example}

\begin{proposition}\label{proprietedeyoneda}
Let $F$ be an object of $\mathcal{C}_T\mathrm{-mod}$. Then the map $$\mathcal{H}om(T,F)\stackrel{\mu^*}{\rightarrow} \mathcal{H}om(\underline{\underline{\mathcal{C}}},F)\cong F$$
induced by $\mu$ is an isomorphism.
\end{proposition}
\begin{proof}
Let us say that a functor is nice if it satisfies the property of the
proposition. That a functor $F=T_a$ is nice is obvious. We will show
that any limit, resp. any colimit, of nice functors is a nice functor. By
the definition of $\mathcal{C}_T\mathrm{-mod}$, this will give the desired result. First, suppose $F$ is a functor of $\mathcal{C}_T\mathrm{-mod}$ which is a limit of nice objects. Write $F=\mathrm{lim}\ F_i$. The bifunctor $\Hom(-,-)$ commutes with limits, so 
$$\Hom(T_a,\mathrm{lim}\ F_i)\cong \mathrm{lim}\ \Hom(T_a,F_i)\cong F(a).$$ 
Secondly, suppose $F$ is a functor of $\mathcal{C}_T\mathrm{-mod}$ which is a colimit of nice objects. We write $F=\mathrm{colim}\ F_i$. Let 
$$\xymatrix@{}{F_i \ar[d]_{f_{i}}& \Hom(T_a,F_i) \ar[d]_{c_i}& \Hom(P_a^{\mathcal{C}},F_i) \ar[d]_{d_i}\\ 
\mathrm{colim}\ F_i& \mathrm{colim}\ \Hom(T_a,F_i)& \mathrm{colim}\ \Hom(P_a^{\mathcal{C}},F_i)}$$ 
be the universal cocones. We obtain the following diagram
$$
\xymatrix@{}{&\Hom(T_a,F_i) \ar[ldd]_{f_{i*}}  \ar@{^{(}->}[r]^{\mu^*} \ar[d]^{c_i}& \Hom(P_a^{\mathcal{C}},F_i) \ar[d]_{d_i} \ar[rdd]^{f_{i*}}&\\
&\mathrm{colim}\ \Hom(T_a,F_i) \ar[r]_{\alpha} \ar[ld]^{\beta} &\mathrm{colim}\ \Hom(P_a^{\mathcal{C}},F_i) \ar[rd]_{\gamma}&\\
 \Hom(T_a,\mathrm{colim}\ F_i) \ar@{^{(}->}[rrr]_{\mu^*}& & &\Hom(P_a^{\mathcal{C}},\mathrm{colim}\ F_i)}
 $$
where $\alpha, \beta$ and $\gamma$ are given by the universal property. We show that $\beta$ is an isomorphism.  The morphism\linebreak
$\mu^*:\Hom(T_a,F_i)\rightarrow \Hom(P_a^{\mathcal{C}},F_i)$
is an isomorphism. It follows that $\alpha$ is also an isomorphism. The standard projective $P_a^{\mathcal{C}}$ commutes with colimits, (\textit{i.e.} $\gamma$ is an isomorphism) so the
cocone 
$\Hom(P_a^{\mathcal{C}},F_i)\stackrel{f_{i*}}{\rightarrow}\Hom(P_a^{\mathcal{C}},\mathrm{colim}\ F_i)$
is also universal. In the diagram, the square, the triangles and the large trapeze commute. We obtain the equalities
$$\mu^* \beta c_i=\mu^*f_{i*}=f_{i*}\mu^*= \gamma d_i \mu^*=\gamma
\alpha c_i\ \ \ .$$ The cocone $\{c_i\}$ being universal, and $\gamma
\alpha$ being an isomorphism, it follows that $\mu^*\beta$ is an
isomorphism, so that $\mu^*$ is surjective (so it's an isomorphism). We deduce that the bottom cocone $f_{i*}$ is
universal and $\beta$ is an isomorphism. Finally one obtains the isomorphisms
$$\Hom(T_a,\mathrm{colim}\ F_i)\cong \mathrm{colim}\ \Hom(T_a,F_i)\cong F(a).$$
\end{proof}
\begin{corollary}\label{CT}
The $\mathcal{C}$-modules $T_a$ form a system of projective generators of $\mathcal{C}_T\mathrm{-mod}$. They are called $T$-standard projectives.
\end{corollary}
\begin{proof}
According to the property of the previous proposition, the functor $$\Hom_{\mathcal{C}_T\mathrm{-mod}}(T_a,-)$$ of $(\mathcal{C}_T\mathrm{-mod})\mathrm{-mod}$ is exact. Hence $T_a$ is projective. Let $F$ an object of $\mathcal{C}_T\mathrm{-mod}$. Let $\lambda \in F(a)$. There is a unique natural transformation $\rho_{\lambda}:T_a \to F$ such that $(\rho_{\lambda})_a(\mu(\textbf{1}_a))=\lambda$. The natural transformation $$\bigoplus_{a, \lambda \in F(a)}\rho_{\lambda}:  \bigoplus_{a, \lambda \in F(a)}T_a \to F$$
is an epimorphism. It follows that $T_a$'s form a system of generators of $\mathcal{C}_T\mathrm{-mod}$.
\end{proof}

\begin{corollary}\label{py}
Let $F$ and $G$ be two objects of $\mathcal{C}_T\mathrm{-mod}$ and $\mathcal{C}_{T^{op}}\mathrm{-mod}$ respectively. Then
\begin{enumerate}
\item $T\otimes_{\mathcal{C}}F\cong F$
\item $G\otimes_{\mathcal{C}}T\cong G$
\end{enumerate}
\end{corollary}
\begin{proof}
The tensor product over $\mathcal{C}$ commutes with colimits. The $T_a$'s are generators so a functor of $\mathcal{C}_T\mathrm{-mod}$ (resp. $\mathcal{C}_{T^{op}}\mathrm{-mod})$ is a colimit of the $T_a$'s (resp. $T^a$'s). We might need to iterate \textit{i.e.} take colimit of objects which are themselves colimit of the $T_a$. It suffices therefore to prove the property $1$ (resp. $2$) for $T_a$ (resp. $T^a$). This is a consequence of proposition \ref{Y1}.
\end{proof}
\begin{remark}
We call the isomorphisms of proposition \ref{proprietedeyoneda} and corollary \ref{py} the Yoneda properties.
\end{remark}
\begin{definition}Let $B$ be a category and $A$ a full subcategory. We say that $A$ is epireflective if the inclusion $ i: A \rightarrow B $ admits a left adjoint $ t: B \rightarrow A $ such that the unit of the adjunction is an epimorphism.  We say that $A$ is coreflective if the inclusion $ i: A \rightarrow B $ admits a right adjoint.
\end{definition}
\begin{proposition}\label{epireflexive}
There is a one to one correspondence between equivalence classes of epireflective and coreflective subcategories of $ \mathcal{C} \mathrm{-mod} $ and isomorphism classes of Yoneda bifunctors for $ \mathcal{C} $.
\end{proposition}
We need the following digression before proving the proposition. A subcategory $\mathcal{B}$ of $\mathcal{A}$ is said dense in $\mathcal{A}$ if every object of $\mathcal{A}$ is a colimit of objects of $\mathcal{B}$ in a canonical way (see \protect{\cite[X.6]{categories}}). Let $\mathcal{G}$ be the full subcategory of $\cmod$ whose objects are the functors $P_a^{\mathcal{C}}\oplus P_a^{\mathcal{C}}$ for all object $a$ of $\mathcal{C}$. 
\begin{lemma}\label{dense} 
The subcategory $\mathcal{G}$ is dense in $\cmod$.
\end{lemma}
\begin{proof}
The lemma is equivalent to the following assertion: The functor 
$$\begin{array}{rcl}\cmod&\rightarrow &\mathcal{G}^{op}-\mathrm{Set}\\
F&\mapsto& \Hom_{\cmod}(-,F)\end{array}$$
is full and faithful where $\mathcal{G}^{op}-\mathrm{Set}$ is the category of covariant functor from $\mathcal{G}^{op}$ to the category of sets (see \protect{\cite[X.6 page 247]{categories}}). We have to prove that the map
$$\phi:\Hom_{\cmod}(F,G)\rightarrow \Hom_{\mathcal{G}^{op}-\mathrm{Set}}(\Hom_{\cmod}(-,F),\Hom_{\cmod}(-,G)),$$
which assigns to a natural transformation $N:F\rightarrow G$ the natural transformation $N_*:\Hom_{\cmod}(-,F)\rightarrow \Hom_{\cmod}(-,G)$ whose components are the postcomposition by $N$, is an isomorphism. We remark that the component of $N_*$ assigned to $P_a^{\mathcal{C}}\oplus P_a^{\mathcal{C}}$ is given by the map $$T_1\oplus T_2 \mapsto N \circ T_1\oplus N \circ T_2$$ where $$T_1\oplus T_2 \in \Hom_{\cmod}(P_a^{\mathcal{C}}\oplus P_a^{\mathcal{C}},F) \cong \Hom_{\cmod}(P_a^{\mathcal{C}}\oplus P_a^{\mathcal{C}},F).$$ We construct an inverse to $\phi$. Let $M$ an element of $\Hom_{\mathcal{G}^{op}-\mathrm{Set}}(\Hom_{\cmod}(-,F),\Hom_{\cmod}(-,G))$. Consider the component 
$$M_{P_a^{\mathcal{C}}\oplus P_a^{\mathcal{C}}}:\Hom_{\cmod}(P_a^{\mathcal{C}}\oplus P_a^{\mathcal{C}},F) \rightarrow \Hom_{\cmod}(P_a^{\mathcal{C}}\oplus P_a^{\mathcal{C}},G)$$
of $M$ and identify $\Hom_{\cmod}(P_a^{\mathcal{C}}\oplus P_a^{\mathcal{C}},F)$ to $\Hom_{\cmod}(P_a^{\mathcal{C}},F)\oplus \Hom_{\cmod}(P_a^{\mathcal{C}},F)$. By the Yoneda isomorphism, a morphism $P_a^{\mathcal{C}}\oplus P_a^{\mathcal{C}}\rightarrow P_a^{\mathcal{C}}\oplus P_a^{\mathcal{C}}$ is represented by a square matrix with coefficient in $k[\Hom_{\mathcal{C}}(a,a)]$. Using the naturality of the component $M_{P_a^{\mathcal{C}}\oplus P_a^{\mathcal{C}}}$ with respect to the morphisms represented by the matrices
$\begin{pmatrix}
0 & \mathbf{1}_a\\
\mathbf{1}_a& 0
\end{pmatrix},
\begin{pmatrix}
 \mathbf{1}_a&0\\
0& 0
\end{pmatrix},
\begin{pmatrix}
0 & \mathbf{1}_a\\
0& 0
\end{pmatrix}
\text{ and }
\begin{pmatrix}
\mathbf{1}_a & \mathbf{1}_a\\
0& 0
\end{pmatrix}
$
one can see that $M_{P_a^{\mathcal{C}}\oplus P_a^{\mathcal{C}}}=M_{P_a^{\mathcal{C}}}\oplus M_{P_a^{\mathcal{C}}}$ where $$M_{P_a^{\mathcal{C}}}:\Hom_{\cmod}(P_a^{\mathcal{C}},F)\rightarrow \Hom_{\cmod}(P_a^{\mathcal{C}},G)$$ is a morphism of $k$-modules which is natural with respect to $a$. By the Yoneda isomorphism, identify $\Hom_{\cmod}(P_a^{\mathcal{C}},F)$ to $F(a)$ then $M_{P_a^{\mathcal{C}}}$
induces a morphism $\widetilde{M}_a:F(a)\rightarrow G(a)$ natural in $a$. The association $M\mapsto \widetilde{M}$ defines an inverse to $\phi$. 
\end{proof}

\begin{proof}[proof of proposition \ref{epireflexive}]
Let $T$ be a Yoneda bifunctor with a surjection $\mu$. We show that the full subcategory $\mathcal{C}_T\mathrm{-mod}$ of $\mathcal{C}\mathrm{-mod}$ is epireflective and coreflective. We construct a left adjoint $t: \mathcal{C}\mathrm{-mod} \rightarrow
\mathcal{C}_T\mathrm{-mod}$ to the inclusion functor $i:\mathcal{C}_T\mathrm{-mod} \xymatrix{ \ar@{^{(}->}[r] &}\mathcal{C}\mathrm{-mod}$.
First, we define $t$ on the subcategory $\mathcal{G}$ of $\cmod$ by $t(P_a^{\mathcal{C}}\oplus P_a^{\mathcal{C}})=T_a\oplus T_a$. Secondly, for $F=\mathrm{colim}\ F_i$ with $F_i \in \text{Ob}\mathcal{G}$, we define $t(F)$ by $t(F)=\mathrm{colim}\ t(F_i)$ (it is well-defined because the colimit is canonical see the paragraph before lemma \ref{dense}). The functor $t$ is a left adjoint to the inclusion $i$.
Indeed for any object $G$ of $\mathcal{C}_T\mathrm{-mod}$ we have a natural isomorphism
 $$\Hom_{\modcat{\mathcal{C}_T}}(t(F),G)\cong \Hom_{\cmod}(F, i(G))$$ obtained as follows:
$$\begin{array}{cclcc}
\Hom_{\modcat{\mathcal{C}_T}}(t(\mathrm{colim}\ F_{i}), G) &\cong &\mathrm{lim}\ \Hom_{\modcat{\mathcal{C}_T}}(t(F_{i}),
G)&&\\
&\cong &\mathrm{lim}\ \Hom_{\cmod}(F_i, i(G)) &&\\
&\cong &\Hom_{\cmod}(\mathrm{colim}\ F_i, i(G))&&\\
&\cong
&\Hom_{\cmod}(F, i(G))
 .\end{array}$$
The second isomorphism above is a consequence of proposition \ref{proprietedeyoneda} and the fullness of $\modcat{\mathcal{C}_T}$ in $\cmod$.
The unit of the adjunction is the natural transformation given for each object $X$ of $\cmod$ by the map $$\varepsilon_X:X=\mathrm{colim}\ F_i\rightarrow it(X)=\mathrm{colim}\ t(F_{i})$$ induced by the surjections $\mu_{F_i}:F_i\rightarrow t(F_i)$. The functor $\mathrm{colim}$ is right exact then $\varepsilon_X$ is surjective. The isomorphism $\Hom_{\modcat{\mathcal{C}_T}}(T_a,G)\cong \Hom_{\cmod}(P_a^{\mathcal{C}}, i(G))$ for each object $G$ of $\modcat{\mathcal{C}_T}$ implies $t(P_a^{\mathcal{C}})\cong T_a$. Now let us prove that the functor $i$ has a right adjoint. It suffices to show that the dual conditions of the special adjoint functor theorem  \protect{\cite[theorem 3.3.4 page 110]{borceux}} are satisfied for $i$. The category $\modcat{\mathcal{C}_T}$ is cocomplete and the functor $i$ preserves colimits by definition of $\modcat{\mathcal{C}_T}$. It remains to verify that $\modcat{\mathcal{C}_T}$ is co-well-powered (the dual notion of well-powered see \protect{\cite[definition 4.1.2]{borceux}}) and has a generating family. The later condition follows from corollary \ref{CT} and the essential smallness of $\mathcal{C}$. The former is a consequence of \protect{\cite[proposition 4.5.15]{borceux}} (notice that for an object in an abelian category there's a one to one correspondence between its subobjects and its quotient objects thus well-powered implies co-well-powered).
Conversely, let $\mathcal{H}$ an epireflective and coreflective subcategory of
$\mathcal{C}\mathrm{-mod}$. Let $t$ be the left adjoint of the inclusion $i:\mathcal{H} \rightarrow \mathcal{C}\mathrm{-mod}$. We denote $\varepsilon$ the unit of the adjunction. The bifunctor 
$$T^\mathcal{H}:\mathcal C^{op}\times\mathcal C \to k\mathrm{-mod}$$
defined by
$$T^\mathcal{H}(a,b)=t(P_{a}^{\mathcal{C}})(b)$$
with the surjection $\mu_{a,b}=\varepsilon_{P_a^{\mathcal{C}}}(b)$
is a Yoneda bifunctor as shown by the isomorphisms
 $$\begin{array}{cclcccc}
\Hom(T^\mathcal{H}_a,T^\mathcal{H}_b)&=&\Hom(tP_a^{\mathcal{C}}, tP_b^{\mathcal{C}})&&&&\\
&\cong&\Hom(P_a^{\mathcal{C}},i(tP_b^{\mathcal{C}}))&\cong&t(P_b^{\mathcal{C}}))(a)&= &T^\mathcal{H}(b,a)\ .\end{array}$$
The first isomorphism sends a natural transformation $F:T^\mathcal{H}_a \rightarrow T^\mathcal{H}_b$ to $F\circ \varepsilon_{P_a^{\mathcal{C}}}$. The second is the Yoneda isomorphism (see properties \ref{propreties}). Since a left adjoint is unique up to isomorphism then the correspondence $\mathcal{H}\mapsto T^\mathcal{H}$ between the equivalence classes of epireflective and coreflective subcategories of $\cmod$ and the isomorphism classes of Yoneda bifunctor for $\mathcal{C}$ is well defined. By the first part of the proof, the subcategory $\modcat{\mathcal{C}_T}$ corresponds to the bifunctor $T$. In order for this correspondence to be one to one, we must prove that $\mathcal{H}$ and $\modcat{\mathcal{C}_{T^\mathcal{H}}}$ are equivalent. Let us observe first that $\mathcal{H}$ is stable (up to isomorphism) under taking limits and colimits (\textit{i.e} a (co)limit in $\cmod$ of a diagram in $\mathcal{H}$ is isomorphic to the (co)limit computed in $\mathcal{H}$ see the proof of \protect{\cite[proposition 3.5.3]{borceux}} and its dual). The adjunction isomorphism and the Yoneda Lemma provide the isomorphism $\Hom_{\mathcal{H}}(T^\mathcal{H}_a,F)\cong F(a)$ for each object $F$ of $\mathcal{H}$. The same reasoning as in corollary \ref{CT} shows that the $T^\mathcal{H}_a$'s form a system of projective generators of  $\mathcal{H}$ and thus each object of $\mathcal{H}$ is a colimit of the $T^\mathcal{H}_a$'s taking into account that we may need to iterate. It follows that $\mathcal{H}$ contains an isomorphic copy of each object of $\modcat{\mathcal{C}_{T^\mathcal{H}}}$ and therefore the two subcategories are equivalent. 
\end{proof}
\section{Functor of duality}
In this section, we fix a Yoneda bifunctor $T$ for $\mathcal{C}$. We recall that $\mathcal{C}_T\mathrm{-mod}$ is the smallest full subcategory of $\mathcal{C}\mathrm{-mod}$    containing the $T_a$'s and stable under limits and colimits. This is an abelian category with enough projectives. The projectives $T_a$, called $T$-standards, form a system of generators of $\mathcal{C}_T\mathrm{-mod}$.
\bull{Definitions.}
Let $R$ be a unitary ring and $M$ a left $R$-module. The dual of $M$ is the right $R$-module $M':=\Hom_{\modcat{R}}(M,R)$. The functor of duality $D$ defined below is a generalization of the duality operator $M\mapsto M'$ where $R$ is replaced by a category $\mathcal{C}$.
\par
We denote $$D:(\mathcal{C}\mathrm{-mod})^{op} \rightarrow \mathcal{C}^{op}\mathrm{-mod}$$ 
the functor that sends a left $\mathcal{C}$-module $F$ to the right $\mathcal{C}$-module $DF:=\mathcal{H}om(F,T).$
We refer to notations~\ref{notations} for the definition of $\mathcal{H}om(F,T)$ \textit{i.e.} for an object $x$ of $\mathcal{C}$, the equality
$$DF(x)=\Hom_{\mathcal{C}\mathrm{-mod}}(F,T_x)$$
where $T_x$ is the functor $T(x,-)$ defines $\mathcal{H}om(F,T)$ on objects.
\par
The functor of duality $D$ extends to the category $(\Ch(\mathcal{C}\mathrm{-mod}))^{op}$ as follows. 
To a complex $X=(X^n,d_X^n)$ of $\Ch(\mathcal{C}\mathrm{-mod})$, the functor 
$$D:(\Ch(\mathcal{C}\mathrm{-mod}))^{op} \rightarrow \Ch(\mathcal{C}^{op}\mathrm{-mod})$$
 associates the complex $D(X)$ defined by 
$$D(X)^n=D(X^{-n})\ \ \text{   and   }\ \  d^n_{D(X)}=(-1)^{n+1}D(d_X^{-(n+1)}).$$
\bull{Behavior of $D$ on projectives of $\mathcal{C}_T\mathrm{-mod}$.}
\begin{example}\label{ex1}
For any object $a$ of $\mathcal{C}$, we have a natural isomorphism
$$D(T_a)\cong T^a.$$ Indeed, one has natural isomorphisms, as a consequence of the Yoneda properties,
 $$D(T_a)(x)=\Hom_{\mathcal{C}\mathrm{-mod}}(T_a,T_x)\cong T_x(a)=T^a(x)$$
for all $x\in \text{Ob}\mathcal{C}$. The equality $D(T_a)=T^a$ follows.
\end{example}
The previous example and the fact that $D$ commutes with limits imply that the functor $D$ restricts to a functor 
$$D:(\mathcal{C}_T\mathrm{-mod})^{op}\rightarrow \mathcal{C}_{T^{op}}\mathrm{-mod}.$$ 
\begin{definition}An object of $\mathcal{C}_T\mathrm{-mod}$ is said of finite type if it is a quotient of a finite sum of $T$-standard projectives.
\end{definition}
\begin{lemma}
Let $P$ be a projective of finite type of $\mathcal{C}_T\mathrm{-mod}$. The object $D(P)$ of $\mathcal{C}_{T^{op}}\mathrm{-mod}$ is a projective of finite type.
\end{lemma}
\begin{proof}
If $P$ is a projective of finite type then $P$ is a direct factor of a finite sum of $T$-standard projectives $F$,
\textit{i.e.} there is a $\mathcal{C}$-module $Q$ such that $P\oplus Q=F$.
By additivity of the functor $D$, it suffices to prove the lemma for $P=T_a$, case corresponding to the example~\ref{ex1}.
\end{proof}
\begin{notations}
We denote $Proj(\mathcal{C}_T\mathrm{-mod})$ the full subcategory of $\mathcal{C}_T\mathrm{-mod}$ whose objects are the projectives of finite type.
\end{notations}
\begin{proposition}\label{equiv}
The restriction $$D: (Proj(\mathcal{C}_T\mathrm{-mod}))^{op} \rightarrow Proj(\mathcal{C}_{T^{op}}\mathrm{-mod})$$ is an equivalence of categories. 
\end{proposition}
\begin{proof}
Let $P$ and $Q$ be two projectives of finite type of $\mathcal{C}_T\mathrm{-mod}$ then the map 
\begin{equation}\label{star}
\Hom_{\mathcal{C}_T\mathrm{-mod}}(P,Q) \stackrel{D}{\rightarrow}\Hom_{\mathcal{C}_{T^{op}}\mathrm{-mod}}(D(Q),D(P)) \tag{$\star$}
\end{equation} 
is a natural isomorphism in two variables. 
Indeed, by additivity of $D$ it suffices to check ~\eqref{star} for $P=T_a$ and $Q=T_b$. 
In this case, the assertion is a consequence of the natural isomorphisms
$$\Hom_{\mathcal{C}_T\mathrm{-mod}}(T_a,T_b)\cong T(b,a) \cong \Hom_{\mathcal{C}_{T^{op}}\mathrm{-mod}}(T^b,T^a).$$
The functor $D$ has an inverse which we denote by $D_{T^{op}}$ defined for an object $G$ of $\mathcal{C}_{T^{op}}\mathrm{-mod}$ by $D_{T^{op}}G= \Hom(G,T^{op})$. 
Indeed, the composition
$$P(x)\cong \Hom_{\mathcal{C}_T\mathrm{-mod}}(T_x,P)\stackrel{D}{\rightarrow}\Hom_{\mathcal{C}_{T^{op}}\mathrm{-mod}}(DP,T^x)\cong D_{T^{op}}\circ DP(x)$$
defines an isomorphism $I\cong D_{T^{op}}\circ D$.
\end{proof}
We define a natural map 
$$\phi_{G}:D(F)\otimes_{\mathcal{C}}G \rightarrow \Hom_{\mathcal{C}_T\mathrm{-mod}}(F,G)$$
 as follows. Let $x\in \text{Ob}\mathcal{C}$.  Then $\phi$ is induced by the composition
$$\begin{array}{ccl}
D(F)(x)\otimes G(x)
&\stackrel{\approx}{\rightarrow}&\Hom_{\mathcal{C}_T\mathrm{-mod}}(F,T_x)\otimes
\Hom_{\mathcal{C}_T\mathrm{-mod}}(T_x,G)\\
& \rightarrow& \Hom_{\mathcal{C}_T\mathrm{-mod}}(F,G)\ .\end{array}$$
\medskip
Let $R$ be a unitary ring, $M$ a right $R$-module and $N$ a left $R$-module. If $M$ is projective of finite type then we have a natural isomorphism \begin{equation}\label{dua0} \Hom_{\modcat{R}}(M,N) \cong M'\otimes_{R}N.\end{equation} 
The following proposition generalizes the latter isomorphism when $R$ is replaced by a category $\mathcal{C}$.

\begin{proposition}\label{dua1}
Let $P$ be a projective of finite type of $\mathcal{C}_T\mathrm{-mod}$. Then the natural  transformation between functors of $(\mathcal{C}_T\mathrm{-mod})\mathrm{-mod}$
\begin{equation} D(P)\otimes_{\mathcal{C}}- \stackrel{\approx}{\rightarrow}\Hom_{\mathcal{C}_T\mathrm{-mod}}(P,-) .\label{di}\end{equation} 
is an isomorphism.
\end{proposition}
\begin{proof}
According to the Yoneda properties, we have the following isomorphisms
$$D(T_a)\cong T^a \text{ and
}\Hom_{\mathcal{C}_T\mathrm{-mod}}(T_a,G)\cong G(a) \cong
T^a\otimes_{\mathcal{C}}G.$$ It follows that the statement is true for $P=T_a$.
If $P$ is a direct summand in a finite sum of $T$-standard projectives, the proposition is derived from the additivity of $D$.
\end{proof}
\bull{Behavior of $D$ on complexes of projectives and the derived functor of $D$.}
\begin{proposition}
Let $X=(X^n,d_X^n)$ be a complex of projectives of finite type of ${{\Ch}}^b(\mathcal{C}_T\mathrm{-mod})$ and $Y=(Y^n,d_Y^n)$ a complex of\linebreak ${{\Ch}}(\mathcal{C}_{T}\mathrm{-mod})$. Then we have a natural isomorphism
$$\Hom(X,Y) \cong D(X)\otimes Y.$$ 
\end{proposition}
\begin{proof}
This is a consequence of the isomorphism of the proposition~\ref{dua1}  and the compatibility of differentials with this isomorphism.
\end{proof}
Recall that the homotopy category $\mathcal{K}^*(\mathcal{A})$
is obtained from the category $\Ch^*(\mathcal{A})$ by
identifying homotopically equivalent morphisms, \textit{i.e.}by setting $$\Hom_{\mathcal{K}(\mathcal{A})}(X,Y):= H^0(\Hom(X,Y)),$$
and that the derived category $\mathcal{D}^*(\mathcal{A})$ is the triangulated category obtained from the homotopy category $\mathcal{K}^*(\mathcal{A})$ by inverting quasi-isomorphisms.
\par
The bifunctors $\Hom(-,-)$ and  $-\otimes-$ induce the derived bifunctors (see \protect{
\cite[page 91 and 98]{Borel}})
$$R\Hom:(\mathcal{D}^-(\mathcal{C}\mathrm{-mod}))^{op}\times \mathcal{D}(\mathcal{C}\mathrm{-mod})\rightarrow \mathcal{D}(k\mathrm{-mod}),$$
and 
$$-\stackrel{L}{\otimes}-:\mathcal{D}^{-}(\mathcal{C}^{op}\mathrm{-mod})\times \mathcal{D}^-(\mathcal{C}\mathrm{-mod})\rightarrow \mathcal{D}(k\mathrm{-mod}).$$
The functor $D$ admits a derived functor
$$RD: (\mathcal{D}^-(\mathcal{C}_T\mathrm{-mod}))^{op} \rightarrow \mathcal{D}^{+}(\mathcal{C}_{T^{op}}\mathrm{-mod}).$$
To calculate the derived dual of a complex $X$ of $\mathcal{D}^-(\mathcal{C}_T\mathrm{-mod})$, 
we choose a projective resolution of $X$ and we apply the functor $D$. 
Let $F$ be an object of $\mathcal{C}_T\mathrm{-mod}$ and $x$ an object of $\mathcal{C}$. 
We have the equality:
$$\operatorname{H}^i(RD(F))(x)=\operatorname{Ext}_{\modcat{\mathcal{C}_T}}^i(F,T_x).$$
 In the sequel, we set $D^i(X)=H^i(RD(X))$.
\begin{definition}
A complex $X$ of $\mathcal{D}(\mathcal{C}_T\mathrm{-mod})$ is perfect if it is  quasi-isomorphic to a bounded complex of finite type projectives.\end{definition}
\begin{proposition}\label{dua3}
Let $X$ and $Y$ be two complexes of $\mathcal{D}^-(\mathcal{C}_T\mathrm{-mod})$. If $X$ is perfect, then we have a natural isomorphism
$$R\Hom(X,Y) \cong RD(X)\stackrel{L}{\otimes} Y.$$ 
\end{proposition}
\begin{proof}
Let $P$ a complex of projectives of finite type of\linebreak$\Ch^b(\mathcal{C}_T\mathrm{-mod})$ quasi-isomorphic to $X$. Then we have the isomorphisms
$$R\Hom(X,Y)= \Hom(P,Y)\cong D(P)\otimes Y=RD(X)\stackrel{L}{\otimes} Y.$$
\end{proof}
\bull{Illustration: Serre duality.}
Let $k$ be a field. For a $k$-vector space $V$, we put $V':=\Hom_k(V,k)$.
\begin{definition}
Let $\mathcal{A}$ be a $k$-linear category. 
A Serre duality for $\mathcal{A}$ is an endofunctor $S:\mathcal{A}\rightarrow \mathcal{A}$
equipped with a natural isomorphism
$$\Phi_{a,b}:\Hom_{\mathcal{A}}(b,Sa) \rightarrow \Hom_{\mathcal{A}}(a,b)'$$
for all objects $a,b$ of $\mathcal{A}$. [Compare to \cite{RV}]
\end{definition}
\begin{remark}\label{unicitedualitedeserre}~~
\begin{enumerate}
\item The functor $S$ is a Serre duality for $\mathcal{A}$ if and only if
the object $Sx$ is an object representing the functor $\Hom_{\mathcal{A}}(x,-)^*$ of $\mathcal{A}^{op}\mathrm{-mod}$ for all object $x$ of $\mathcal{A}$.
 \item The Yoneda embedding
$$\begin{array}{crcl}
&\mathcal{A} &\rightarrow &\mathcal{A}^{op}\mathrm{-mod}\\
&c&\mapsto& \Hom_{\mathcal{A}}(-,c)
\end{array}$$
is fully faithful. Hence the uniqueness of a Serre duality up to isomorphism.
\item The Yoneda lemma provides a natural isomorphism
$$\Hom_{\mathcal{A}^{op}\mathrm{-mod}}(\Hom_{\mathcal{A}}(-,Sa),\Hom_{\mathcal{A}}(a,-)')\cong \Hom_{\mathcal{A}}(a,Sa)'.$$
We call fundamental class the element $\Phi_a \in \Hom_{\mathcal{A}}(a,Sa)'$ corresponding to the isomorphism $\Phi_{a,b}$. The isomorphism $\Phi_{a,b}$ is right adjoint (in the sense of the theory of bilinear forms) of the bilinear form 
$$\Hom_{\mathcal{A}}(a,b)\times \Hom_{\mathcal{A}}(b,Sa)\rightarrow \Hom_{\mathcal{A}}(a,Sa)\stackrel{\Phi_a}{\rightarrow}k.$$
\end{enumerate}
\end{remark}
Let $T$ a Yoneda bifunctor for $\mathcal{C}$. Assume that $\mathcal{C}$ is $k$-linear.
\begin{notations}~~
\begin{itemize}
\item We denote $\mathcal{D}_{\mathrm{perf}}(\mathcal{C}_T\mathrm{-mod})$ the full subcategory of $\mathcal{D}(\mathcal{C}_T\mathrm{-mod})$ whose objects are perfect complexes.
\item We denote $F'$ the postcomposition of $F$ by the duality operation $V\mapsto V':=\Hom_k(V,k)$.
\end{itemize}
\end{notations}
\begin{proposition}\label{dualitedeserre}
Let $X$ and $Y$ be two complexes of  $\mathcal{D}_{\mathrm{perf}}(\mathcal{C}_T\mathrm{-mod})$. We have a natural isomorphism
$$R\Hom(X,Y)'\cong R\Hom(Y, \text{$RD(X)'$}).$$
In particular, the functor $RD(-)'$ is a Serre duality for \linebreak$\mathcal{D}_{\mathrm{perf}}(\mathcal{C}_T\mathrm{-mod})$.
\end{proposition}
\begin{proof}
Let $P$ a complex of finite type projectives of \linebreak$\Ch^b(\mathcal{C}_T\mathrm{-mod})$ quasi-isomorphic to $X$. Then we have the isomorphisms
$$R\Hom(X,Y)'\cong (D(P)\otimes Y)'\cong \Hom(Y, \text{$D(P)'$}).$$
The second one is a consequence of the adjunction characterizing the tensor product.
The functor $F\mapsto F'$ is exact so $D(P)'$ is a bounded injective resolution of $D(X)'$. We deduce an isomorphism
$$R\Hom(X,Y)'\cong R\Hom(Y, \text{$RD(X)'$}).$$
\end{proof}
\section{Duality functors}
Let $T$ a Yoneda bifunctor fo $\mathcal{C}$. We introduce the concept of duality functors in $\mathcal{C}_T\mathrm{-mod}$. This concept is measured by the action of the functor of duality
$$D:(\mathcal{C}_T\mathrm{-mod})^{op}\rightarrow \mathcal{C}_{T^{op}}\mathrm{-mod}$$
on some objects of $\mathcal{C}_T\mathrm{-mod}$.\\
\bull{Definitions and first properties.}
\begin{definition} A $\mathcal{C}_T$-module is of type
  $FP_{\infty}$
(resp. of type $FP$) if it has a resolution (resp. finite resolution) by finite type projectives of $\mathcal{C}_T\mathrm{-mod}$.
\end{definition}
\begin{definition}
A left $\mathcal{C}_T$-module $F$ is said of projective dimension $n$ if for all $\mathcal{C}_T$-module $G$, and all $m > n$, we have the equality $\operatorname{Ext}_{\cmod}^m\left (F,G\right ) =
0$, and there is a $\mathcal{C}$-module $G$ such that we have $\operatorname{Ext}_{\modcat{\mathcal{C}_T}}^n\left (F,G\right ) \neq 0.$
\end{definition}
\begin{notations}
Let $X$ an object of $\mathcal{C}_T\mathrm{-mod}$ and $n$ an integer. We denote $X[n]$ the complex of $\mathcal{D}\left (\mathcal{C}_{T}\mathrm{-mod}\right)$ defined by $X[n]^i=0$ if $i\neq n$ and $X[n]^n=X$.
\end{notations}
\begin{definition}
We call an Eilenberg-Mac Lane functor of degree $n$ an object $F$ of $\mathcal{C}_T\mathrm{-mod}$ such that we have an isomorphism 
$$RD\left (F\right )\cong D^n\left (F\right )[n]$$
in $\mathcal{D}^+\left (\mathcal{C}_{T^{op}}\mathrm{-mod}\right )$
(Recall that $D^n\left (F\right )=H^n(RDF)$).
\end{definition}
\begin{theorem}\label{dualite}
Let $F$ a $\mathcal{C}_T$-module of type $FP$. The following conditions are equivalent
\begin{enumerate}
\item There exists an integer $n$ such that $F$ is an Eilenberg-Mac Lane functor of degree $n$.
\item There exists an integer $n$ such that $$\operatorname{Ext}_{\modcat{\mathcal{C}_T}}^i\left (F,T_x\right )=0$$
for all $i\neq n$ and all $x \in \text{Ob}\mathcal{C}$.
\item There exists an integer $n$ and an isomorphism $$\operatorname{Ext}_{\modcat{\mathcal{C}_T}}^i\left (F,-\right )\cong \operatorname{Tor}^{\modcat{\mathcal{C}_T}}_{n-i}\left (D^n\left (F\right ),-\right ).$$
\end{enumerate}
The integer $n$ is the projective dimension of $F$ in the category $\mathcal{C}_{T}\mathrm{-mod}$.
\end{theorem}
\begin{proof}The equivalence between condition $1$ and condition $2$ is clear. We show that the third condition implies the second one. Let $x$ be an object of
  $\mathcal{C}$. Evaluating the isomorphism of condition $3$ at $T_x$, we obtain $$\operatorname{Ext}_{\modcat{\mathcal{C}_T}}^i\left (F,T_x \right
  )\cong \operatorname{Tor}^{\modcat{\mathcal{C}_T}}_{n-i}\left (D^n\left
  (F\right ),T_x\right ).$$ Since $T_x$ is projective in $\mathcal{C}_{T}\mathrm{-mod}$ then we have for $i\neq n$ the equality
  $\operatorname{Tor}^{\modcat{\mathcal{C}_T}}_{n-i}\left (D^n\left (F\right
  ),T_x\right )=0$ and hence $\operatorname{Ext}_{\modcat{\mathcal{C}_T}}^i\left (F,T_x \right)=0$ for $i\neq n$. We show that condition $1$ implies condition $3$. The
  $\mathcal{C}_T$-module $F$ is of type $FP$ so $F$ is perfect and since $RD\left (F\right )\cong D^n\left (F\right )[n]$ we have by proposition \ref{dua3} the isomorphism $$R\Hom \left (F,-\right )
  \cong D^n\left (F\right )[n]\stackrel{L}{\otimes} -.$$ Taking cohomology, we obtain an isomorphism $\operatorname{Ext}_{\modcat{\mathcal{C}_T}}^i\left (F,-\right )\cong \operatorname{Tor}^{\modcat{\mathcal{C}_T}}_{n-i}\left (D^n\left (F\right ),-\right )$.
\end{proof}
\begin{definition}
A functor $F$ which satisfies one of the conditions of the theorem is called a duality functor and the functor $D^n(F)$ is the dualizing functor of $F$.
\end{definition}
\begin{remark}\label{dualiteremark}
Suppose $k$ is a field. Then the third condition of the theorem is equivalent to the following.
\begin{enumerate}
\item[(3')]\emph{ There exists an integer $n$ and an isomorphism}
$$\operatorname{Ext}_{\modcat{\mathcal{C}_T}}^i\left (F,-\right )'\cong \operatorname{Ext}_{\modcat{\mathcal{C}_T}}^{n-i}\left (-,\text{$D^n\left (F\right )'$}\right ).$$
\end{enumerate}
\end{remark}
\begin{proposition}\label{opdualite}
Let $F$ a duality functor of $\mathcal{C}_T\mathrm{-mod}$ of projective dimension $n$. The functor $D^n(F)$ is a duality functor of $\mathcal{C}_{T^{op}}\mathrm{-mod}$ and $F$ is its dualizing functor. 
\end{proposition}
\begin{proof}
The equivalence 
$$D_{T}: (Proj(\mathcal{C}_T\mathrm{-mod}))^{op} \rightarrow Proj(\mathcal{C}_{T^{op}}\mathrm{-mod})$$
of proposition~\ref{equiv}, with inverse the functor 
$$D_{T^{op}}:Proj(\mathcal{C}_{T^{op}}\mathrm{-mod})\rightarrow (Proj(\mathcal{C}_T\mathrm{-mod}))^{op},$$
passes to quotient and we obtain the equivalence  $$RD_T:\mathcal{K}^b(Proj(\mathcal{C}_T\mathrm{-mod}))^{op} \rightarrow \mathcal{K}^b(Proj(\mathcal{C}_{T^{op}}\mathrm{-mod}))$$
at the level of homotopy categories. The composition respectively on the left and on the right by the equivalences 
$$\mathcal{D}_{\mathrm{perf}}(\mathcal{C}_T\mathrm{-mod}) \cong \mathcal{K}^b(Proj(\mathcal{C}_{T}\mathrm{-mod}))
$$
$$
\mathcal{K}^b(Proj(\mathcal{C}_{T^{op}}\mathrm{-mod})) \cong \mathcal{D}_{\mathrm{perf}}(\mathcal{C}_{T^{op}}\mathrm{-mod})$$
provides the equivalence of categories 
$$RD_T: (\mathcal{D}_{\mathrm{perf}}(\mathcal{C}_T\mathrm{-mod}))^{op} \rightarrow \mathcal{D}_{\mathrm{perf}}(\mathcal{C}_{T^{op}}\mathrm{-mod})$$
with inverse the functor $RD_{T^{op}}$. 
The isomorphism $RD_{T^{op}}(RD_{T}F)\cong F$ implies $RD_{T^{op}}(D^nF)\cong F[n]$. It follows that $D^n(F)$ is a duality functor of $\mathcal{C}_{T^{op}}\mathrm{-mod}$ and $F$ is its dualizing functor. 
\end{proof}
\section{Illustration: strict polynomial functors}
We illustrate these concepts in homology of strict polynomial functors in the sense of Friedlander and Suslin~\cite{FS}. 
The study of extensions groups is well developed this later years (see eg~\cite{Touze2}). 
For a finite field $k$, let $\operatorname{Rep}\Gamma^d_k$ be the category of strict polynomial functors of degree $d$. 
The category $\operatorname{Rep}\Gamma^d_k$ is equipped with a Ringel duality $\theta$ introduced by M.Cha\l upnik~\cite{Chalupnik} and A.Touz\'e~\cite{Touze1}: 
the functor Ringel duality $\theta:=R\mathcal{H}om(\Lambda^d,-):\mathcal{D}^b(\operatorname{Rep}\Gamma^d_k) \rightarrow \mathcal{D}^b(\operatorname{Rep}\Gamma^d_k)$ is an equivalence of categories with inverse the functor
$\theta^{-1}:X\mapsto  \theta(X^\#)^\#$
(see \protect{\cite[corollary 2.4]{Chalupnik}}, \protect{\cite[theorem 3.8]{Touze1}}, \protect{\cite[theorem 4.9]{Krause}}).

H. Krause  in~\cite{Krause} showed that the equivalence $\theta^{-2}$ is a Serre duality for $\mathcal{D}_{\mathrm{perf}}(\operatorname{Rep}\Gamma^d_k)$.
 
This phenomenon of Serre duality for the category $\operatorname{Rep}\Gamma^d_k$ fits in the general framework of the duality presented in this article. 
On one hand, our theory of duality allows to equip the category of strict polynomial functors $\operatorname{Rep}\Gamma^d_k$ with a functor of duality $D$ which in turn provides a Serre duality $RD(-)'$ for $\mathcal{D}_{\mathrm{perf}}(\operatorname{Rep}\Gamma^d_k)$.
On the other hand, we deduce from the uniqueness of a Serre duality that $RD(-)'$ and the Serre duality described by Krause coincide \textit{i.e.} the equivalence of categories $\theta^{-2}$ is isomorphic to $RD(-)'$.
\bull{Serre duality for $\operatorname{Rep}\Gamma^d_k$.}
We follow the notations of \cite{Krause}. We  refer to~\cite{pirashvili} for more details on strict polynomial functors. 
\par\smallskip
Let $\text{P}_k$ the category of $k$-vector spaces of finite dimension and 
$\Gamma^d: \text{P}_k \rightarrow \text{P}_k$ the divided power functor defined by
$$\Gamma^d (V)=\left(V^{\otimes^d}\right)^{\mathfrak{S}_d}$$
[the symmetric group $\mathfrak{S}_d$ acting by permutation of the factors of the tensor product]. 
We denote $\Gamma^d \text{P}_k$ the category whose objects are finite dimensional $k$-vector spaces, whose morphisms are defined by 
$$\Hom_{\Gamma^d \text{P}_k}(V,W)=\Gamma^d\Hom_k(V,W)~~~~(\forall V,W  \in \text{P}_k)$$
and composition of morphisms is defined \textit{via} the identification 
$$\Gamma^d\Hom_k(V,W)=\Hom_k \left(V^{\otimes^d},W^{\otimes^d}\right)^{\mathfrak{S}_d}.$$
The category of strict polynomial functor of degree $d$ is the category of $k$-linear functors from $\Gamma^d \text{P}_k$ to the category of $k$-modules. 
We denote $\operatorname{Rep}\Gamma^d_k$ this category and $\operatorname{Rep}((\Gamma^d_k)^{op})$ the category of contravariant $k$-linear functors from $\Gamma^d \text{P}_k$ to the category of $k$-modules.

For example, the classical functors of tensor power $\otimes^d$, of symmetric power $S^d$, of external power $\Lambda^d$ and of divided power $\Gamma^d$ are strict polynomial functors of degree $d$.

 The category $\operatorname{Rep}\Gamma^d_k$ is an epireflective and coreflective subcategory of $\modcat{\Gamma^d \text{P}_k}$ associated to the Yoneda bifunctor $L_{\Gamma^d \text{P}_k}$ defined in example~\ref{yonedaexemple}. 
 This is an abelian category. The standard projectives $\Gamma^{d,V}=\Hom_{\Gamma^d \text{P}_k}(V,-)$  form a system of projective generators of $\operatorname{Rep}\Gamma^d_k$. 
 We recall that we have a functor of duality
$$\begin{array}{crcl}
D:&(\operatorname{Rep}\Gamma^d_k)^{op} &\rightarrow &\operatorname{Rep}((\Gamma^d_k)^{op})\\
&F&\mapsto& \mathcal{H}om(F,L_{\Gamma^d \text{P}_k}).
\end{array}$$
\begin{proposition}
The functor $RD(-)'$ is a Serre duality for $\mathcal{D}_{\mathrm{perf}}(\operatorname{Rep}\Gamma^d_k)$.
\end{proposition}
\begin{proof}
This is a rewrite of the proposition \ref{dualitedeserre} in the context of strict polynomial functors.
\end{proof}
\begin{notations}
Let $F$ be an object of $\operatorname{Rep}\Gamma^d_k$ and $V$ an object of $\text{P}_k$.
\begin{itemize} 
\item We denote $\Lambda^{d,V}$ the precomposition of $\Lambda^d$ by $\Hom(V,-)$. 
\item We denote $^\#$ the functor 
$$\begin{array}{cccc}
^\#:&(\operatorname{Rep}\Gamma^d_k)^{op} &\rightarrow&
  (\operatorname{Rep}\Gamma^d_k)\\
& F&\mapsto &F^\#\end{array}$$
defined by $F^\#(V)=F(V')'$.
\item We denote $\mathcal{H}om(\Lambda^d,-)$ the functor 
$$\begin{array}{crcl}
\mathcal{H}om(\Lambda^d,-):&\operatorname{Rep}\Gamma^d_k &\rightarrow &\operatorname{Rep}\Gamma^d_k\\
&F&\mapsto& \mathcal{H}om(\Lambda^d,F)
\end{array}$$
where $\mathcal{H}om(\Lambda^d,F)$ is the functor $V\mapsto \Hom_{\operatorname{Rep}\Gamma^d_k}(\Lambda^{d,V},F)$.
\end{itemize}
\end{notations}
\begin{theorem}[\protect{\cite[theorem 3.8]{Touze1}, \cite[theorem 4.9]{Krause}}]\label{dualitederingel}
The functor Ringel duality 
$$
\theta:=R\mathcal{H}om(\Lambda^d,-):
\mathcal{D}^b(\operatorname{Rep}\Gamma^d_k) \rightarrow \mathcal{D}^b(\operatorname{Rep}\Gamma^d_k)
$$ 
is an equivalence of categories with quasi-inverse the functor
$\theta^{-1}:X\mapsto  \theta(X^\#)^\#$.
\end{theorem}
\begin{proposition}\label{relationthetaetd}The endofunctors $RD(-)'$ and $\theta^{-2}$ of $\mathcal{D}_{\mathrm{perf}}(\operatorname{Rep}\Gamma^d_k)$ are isomorphic.
\end{proposition}
\begin{proof}
H.Krause showed that the functor $\theta^{-2}$ is a Serre duality for $\mathcal{D}_{\mathrm{perf}}(\operatorname{Rep}\Gamma^d_k)$ (see \protect{\cite[corollary 5.5]{Krause}}). 
The proposition follows from the uniqueness of a Serre duality up to an isomorphism (see remark \ref{unicitedualitedeserre}).
\end{proof}
\bull{Example of a duality functor in $\operatorname{Rep}\Gamma^d_k$.}
Let $k$ be a field of characteristic $p$. Let $r$ and $s$ be two positive integers. We put $d=sp^r$. 
 The functor Frobenius twist 
$$\operatorname{I}^{\left (r\right )}:V\mapsto V^{(r)}=k<v^{\otimes p^r},v\in V> \subset S^{p^r}\left (V\right )$$
is a strict polynomial functor of degree $p^r$.
The precomposition by $\operatorname{I}^{\left (r\right )}$ of a strict polynomial functor $F$ of degree $s$ 
define a strict polynomial functor $\operatorname{I}^{\left (r\right )}\circ F$ of degree $sp^r$. 
We put $F^{(r)}:=\operatorname{I}^{\left (r\right )}\circ F$.\\
Let $\mathfrak{s}=\left (s_1,\cdots, s_k\right )$ a partition of $s$. For $X=S$ (resp. $X=\Gamma, X=\Lambda$), we set $X^{\mathfrak{s}}= \otimes_{i=1}^{k}X^{s_i}$. Recall that $X^{\mathfrak{s}}$ is then a strict polynomial functor of degree $s$.
\begin{proposition}\label{exempledualitepolynomiauxstrict}
Let $\mathfrak{s}=\left (s_1,\cdots, s_k\right )$ a partition of $s$. Then the functor $\Gamma^{\mathfrak{s}(r)}$ is a duality functor. In particular
$$\text{$RD$}\left (\Gamma^{\mathfrak{s}(r)} \right )'\cong S^{\mathfrak{s}(r)}[2s\left (p^r-1\right )]  $$
and we have a natural isomorphism
$$\operatorname{Ext}_{\operatorname{Rep}\Gamma^d_k}^i\left (\Gamma^{\mathfrak{s}(r)} ,-\right )'\cong \operatorname{Ext}_{\operatorname{Rep}\Gamma^d_k}^{2s\left (p^r-1\right )-i}\left (-,S^{\mathfrak{s}(r)} \right ).$$
\end{proposition}
\begin{proof}
For the proof of the proposition, we make use of the following properties of the Ringel duality $\theta$ (see theorem~\ref{dualitederingel}) for strict polynomial functors studied in the article \cite{Touze1}.
\begin{enumerate}
\item  We have isomorphisms $\theta(S^{\mathfrak{s}})\cong\Lambda^{\mathfrak{s}}$ and $\theta(\Lambda^{\mathfrak{s}})\cong\Gamma^{\mathfrak{s}}$ (see \protect{\cite[lemma 3.6]{Touze1}}).
\item  Let $C \in \mathcal{D}^b(\operatorname{Rep}\Gamma^d_k)$. We have natural isomorphisms
$$\theta^n(C \circ \operatorname{I}^{(r)})\cong \left( \theta^n(C)\circ \operatorname{I}^{(r)} \right)[dn(p^r-1)]$$
and
$$\theta^n(\operatorname{I}^{(r)}\circ C)\cong \left( \operatorname{I}^{(r)} \circ \theta^n(C) \right)[dn(p^r-1)]$$
(see \protect{\cite[proposition 6.6]{Touze1}}).
\end{enumerate}
The proposition \ref{relationthetaetd} and the definition of the inverse of $\theta$ in theorem \ref{dualitederingel} imply 
$$RD\left (\Gamma^{\mathfrak{s}} \circ \operatorname{I}^{\left (r\right )}\right )'
\cong \theta^{-2}(\Gamma^{\mathfrak{s}} \circ \operatorname{I}^{\left (r\right )})
\cong \theta^2(\Gamma^{\mathfrak{s}} \circ \operatorname{I}^{\left (r\right )})^{\#})^{\#}.$$
The above properties of $\theta$ and the equalities 
$(F\circ G)^{\#}=F^{\#}\circ G^{\#},\Gamma^{\mathfrak{s}\#}=S^{\mathfrak{s}}$ and $\operatorname{I}^{(r)\#}=\operatorname{I}^{(r)}$
allow to rewrite this isomorphism as
$$RD\left (\Gamma^{\mathfrak{s}} \circ \operatorname{I}^{\left (r\right )}\right )'
\cong (\theta^2(S^{\mathfrak{s}}) \circ \operatorname{I}^{\left (r\right )})^{\#}[2s\left (p^r-1\right )]
\cong  S^{\mathfrak{s}} \circ \operatorname{I}^{\left (r\right )}[2s\left (p^r-1\right )].
$$
The functor  $\Gamma^{\mathfrak{s}} \circ \operatorname{I}^{\left (r\right )}$ is then an Eilenberg-Mac Lane functor of degree $2s\left (p^r-1\right )$, the isomorphism between extensions groups follows from remark \ref{dualiteremark}.
\end{proof}

\section{Tits buildings are duality categories}
Let $\mathcal{C}$ be a small category. In this section $k$ is the abelian group $\mathbb Z$. Recall that $\mathcal{C}\mathrm{-mod}$ is the category of covariant functors from  $\mathcal{C}$ to the category of abelian groups.  The functor of duality $$D:(\mathcal{C}\mathrm{-mod})^{op} \rightarrow \mathcal{C}^{op}\mathrm{-mod}$$ which associates to a left $\mathcal{C}$-module $F$ the right $\mathcal{C}$-module $\mathcal{H}om(F,\underline{\underline{\mathcal{C}}})$ sends a standard projective $P_x^{\mathcal{C}}$ of $\mathcal{C}\mathrm{-mod}$ to the standard projective $P^x_{\mathcal{C}}$ of $\mathcal{C}^{op}\mathrm{-mod}$. A functor $F$ of $\mathcal{C}\mathrm{-mod}$ is a duality functor if $RDF\cong D^n(F)[n]$ in $\mathcal{D}^+(\mathcal{C}^{op}\mathrm{-mod} )$ for some integer $n$ \textit{i.e.} given a projective resolution of type $FP$ of $F$, the cohomology of the complex of projectives obtained after applying the functor of duality $D$ is concentrated in degree $n$. In this section we are interested in small categories $\mathcal{C}$ such that the constant functor $\underline{\mathbb Z}$ is a duality functor. Our main result is corollary~\ref{immeubledetits} where we show that spherical Tits buildings are duality categories.

\noindent
\bull{Definitions and first properties.}
\medskip
\begin{definition} A small category $\mathcal{C}$ is a duality category if the $\mathcal{C}$-module $\underline{\mathbb{Z}}$ is a duality functor. In this case we have a natural isomorphism
$$\operatorname{Ext}_{\cmod}^i\left (\underline{\mathbb{Z}},-\right )\cong \operatorname{Tor}^{\cmod}_{n-i}\left (D^n\left (\underline{\mathbb{Z}}\right ),-\right )$$
for some integer $n$. The right $\mathcal{C}$-module $D^n\left (\underline{\mathbb{Z}}\right )$ is called the dualizing functor of $\mathcal{C}$. 
\end{definition}

\begin{example}
Let $\mathcal{C}$ the category $\xymatrix{x \ar@/^/[r]^\alpha\ar@/_/[r]_\beta & y}.$ The following complex of left $\mathcal{C}$-modules 
\begin{equation}0\longrightarrow P_y^{\mathcal{C}} \stackrel{\alpha ^* -\beta^*}{\longrightarrow} P_x^{\mathcal{C}} \label{c1}\end{equation}
is a finite projective resolution of the constant functor $\underline{\mathbb{Z}}$.  Applying the functor of duality $D$, we obtain the complex
\begin{equation}
P^y_{\mathcal{C}} \stackrel{\alpha _* -\beta_*}{\longleftarrow} P^x_{\mathcal{C}} \longleftarrow 0.\label{c2}\end{equation}
We check that it is a finite projective resolution of the right $\mathcal{C}$-module $\underline{\mathbb{Z}}$. It follows that $\underline{\mathbb{Z}}$ is a duality functor and $D(\underline{\mathbb{Z}})\cong \underline{\mathbb{Z}}[1]$.
\end{example}

\begin{example}
Let $\mathcal{C}$ the following category

$$\xymatrix{
4 \ar[rd]_a &0 \ar[rd]^{^{^d}} \ar[d]_b &1 \ar[ld]_{^{^c}} \ar[d]^e\\
&2  &3}$$
The complex of $\mathcal{C}$-modules
$$0\rightarrow P_2^{\mathcal{C}}\oplus P_2^{\mathcal{C}} \oplus P_3^{\mathcal{C}}
\stackrel{\phi}{\rightarrow}P_0^{\mathcal{C}} \oplus P_1^{\mathcal{C}} \oplus P_4^{\mathcal{C}}
$$where $\phi$ is given by the matrix
 $\left(\begin{array}{ccc}
	b^*&b^*&d^*\\
	-c^*&0&-e^*\\
	0&-a^*&0
\end{array}\right)$
is a projective resolution of $\underline{\mathbb{Z}}$. We apply the functor of duality $D$ to this resolution, we obtain the complex
$$P^2_{\mathcal{C}} \oplus P^2_{\mathcal{C}} \oplus P^3_{\mathcal{C}} \stackrel{\phi_*}{\leftarrow} P^0_{\mathcal{C}} \oplus P^1_{\mathcal{C}} \oplus P^4_{\mathcal{C}} \leftarrow 0$$
and we check that it's a projective resolution of the right $\mathcal{C}$-module $D^1(\underline{\mathbb{Z}})$, given on objects by $$D^1(\underline{\mathbb{Z}})(0)=D^1(\underline{\mathbb{Z}})(1)=D^1(\underline{\mathbb{Z}})(2)=\mathbb{Z}^2 \text{ and } D^1(\underline{\mathbb{Z}})(3)=D^1(\underline{\mathbb{Z}})(4)=\mathbb Z.$$
\end{example}
\medskip

To see that a small category $\mathcal{C}$ is a duality category, it suffices to calculate for any object $x$ of $\mathcal{C}$ the extensions groups
$$\operatorname{Ext}_{\cmod}^*(\underline{\mathbb{Z}},P_x^{\mathcal{C}}).$$ If these groups are trivial except for some degree $n$ (which does not depend on $x$) then $\mathcal{C}$ is a duality category.
The purpose of what follows is to simplify the calculations of these groups.

\begin{recalls}[Cohomology of a small category]
Our basic reference for these recalls is chapter $1$ of \cite{Quillen}. More details can be found in \cite{xu}.

The cohomology of a small category $\mathcal{C}$ with coefficient in a functor $F$ is defined by $\operatorname{H}^*(\mathcal{C};F):=\operatorname{Ext}_{\cmod}^*(\underline{\mathbb{Z}},F) ,$ and the homology with coefficient in a functor $F$ is defined by
$\operatorname{H}_*(\mathcal{C};F):=\operatorname{Tor}^{\cmod}_*(\underline{\mathbb{Z}},F).$

To a small category $\mathcal{C}$, we associate a simplicial set $\mathcal{N}\mathcal{C}$ called the nerve of $\mathcal{C}$ (see \cite{categories}) with $\mathcal{N}_n\mathcal{C}$ is the set of chains of morphisms of $\mathcal{C}$ of length $n$. The face operators for $i=0\cdots n$ are given by 
$$d_i^n(x_0\rightarrow \cdots \rightarrow x_n)=x_0\rightarrow \cdots \hat{x}_{i} \cdots \rightarrow x_n,$$
where $x_0\rightarrow \cdots \hat{x}_{i} \cdots \rightarrow x_n$ is the chain of morphisms of length $n-1$ obtained from $x_0\rightarrow \cdots \rightarrow x_n$ replacing the subchain $x_{i-1}\rightarrow x_i \rightarrow x_{i+1}$ by its composition. The $d_i^n$'s extend by linearity to give maps of abelian groups $\mathbb Z[\mathcal{N}_n\mathcal{C}]\rightarrow \mathbb Z[\mathcal{N}_{n-1}\mathcal{C}]$ that we still denote by $d_i^n$. The homology $H_*(\mathcal{N}\mathcal{C},\mathbb Z)$ of $\mathcal{N}\mathcal{C}$ with integral coefficient is the homology of the chain complex $\mathbb Z[\mathcal{N}_*\mathcal{C}]$ with differentials $d_n=\sum_{0}^{n}(-1)^id_i^n$. The cohomology $H^*(\mathcal{N}\mathcal{C};\mathbb Z)$ is the cohomology of the cochain complex $\Hom(\mathbb Z[\mathcal{N}_*\mathcal{C}];\mathbb Z)$. 

We denote $|\mathcal{C}|$ the geometric realization of the simplicial set 
$\mathcal{N}\mathcal{C}$, this is a topological space constructed from $\mathcal{N}\mathcal{C}$. We have the identifications $H_*(\mathcal{N}\mathcal{C};\mathbb Z)\cong H_*(|\mathcal{C}|;\mathbb Z)$ and $H^*(\mathcal{N}\mathcal{C};\mathbb Z)\cong H^*(|\mathcal{C}|;\mathbb Z)$ (see \protect{\cite[example 8.2.3]{weibel}}). There is a bijection between $\mathcal{N}_n\mathcal{C}$ and $\mathcal{N}_n\mathcal{C}^{op}$ that induces an homeomorphism between $|\mathcal{C}|$ and $|\mathcal{C}^{op}|$. We refer to \cite{simplicial} for more details on simplicial sets, their realization and homology.

A small category $\mathcal{C}$ is said contractible if $|\mathcal{C}|$ has the same homotopy type of a point. A category with an initial or a terminal object is contractible. Let $\mathcal{C}$ a small category. The comma category or the fiber category $\mathcal{C}|x$ of objects over $x$ is the category whose objects are morphisms $c\stackrel{\alpha}{\rightarrow} x$ in $\mathcal{C}$ and a morphism from $c\stackrel{\alpha}{\rightarrow} x$ to $c'\stackrel{\beta}{\rightarrow} x$ is a morphism $c\stackrel{\gamma}{\rightarrow}c$ in $\mathcal{C}$ such that $\beta \circ \gamma=\alpha$ (see \protect{\cite[Chapter II.6]{categories}}). The identity morphism of $x$ is a terminal object of $\mathcal{C}|x$ hence $\mathcal{C}|x$ is contractible for each $x\in \text{Ob}(\mathcal{C})$.

We define a complex $\mathcal{B}_{*}^{\mathcal{C}}$ of $\mathcal{C}$-module by
   $$\mathcal{B}_{*}^{\mathcal{C}}:=\mathbb Z[\mathcal{N}_*(\mathcal{C}|-)].$$
 Since the fiber categories $\mathcal{C}|x$ are contractible for each $x\in \text{Ob}(\mathcal{C})$, this complex is a resolution of the constant functor $\underline{\mathbb{Z}}$.
 
 To give a chain of length $n$ in $\mathcal{N}_n(\mathcal{C}|x)$ is equivalent to give a chain $x_0\rightarrow \cdots \rightarrow x_n$ of length $n$ in  $\mathcal{N}_n\mathcal{C}$ and a morphism $x_n\rightarrow  x$. This correspondence induces  an isomorphism
  $$\mathcal{B}_{n}^{\mathcal{C}}\cong  \bigoplus_{x_0\rightarrow \cdots \rightarrow x_n \in \mathcal{N}_n\mathcal{C}}P^{\mathcal{C}}_{x_n}.$$
It follows that $\mathcal{B}_{*}^{\mathcal{C}}$ is a projective resolution of the constant functor $\underline{\mathbb{Z}}$ (called Bar resolution in \cite{xu}) and we have $\operatorname{H}^*(\mathcal{C};F)=\operatorname{H}^*(\Hom(\mathcal{B}_{*}^{\mathcal{C}},F))$. Moreover, the isomorphism $\Hom(\mathcal{B}_{*}^{\mathcal{C}},\underline{\mathbb{Z}}))\cong \Hom(\mathbb Z[\mathcal{N}_*\mathcal{C}];\mathbb Z)$ induces an isomorphism  $\operatorname{H}^*(\mathcal{C};\underline{\mathbb{Z}})\cong H^*(|\mathcal{C}|,\mathbb Z)$. 

The same reasoning replacing $\mathcal{C}$ by the opposite category $\mathcal{C}^{op}$ allows us to explicit a projective resolution of the constant $\mathcal{C}^{op}$-module $\underline{\mathbb{Z}}$ and to identify  $\operatorname{H}_*(\mathcal{C};\underline{\mathbb{Z}})$ to $H_*(|\mathcal{C}|,\mathbb Z)$.

The reduced (co)homology $\widetilde{\operatorname{H}}(\mathcal{C},-)$ of a small category $\mathcal{C}$ is the (co)homology calculated when replacing the complex $\mathcal{B}_{*}^{\mathcal{C}}$ by the augmented complex $\mathcal{B}_{*}^{\mathcal{C}}\rightarrow \underline{\mathbb{Z}}$.

Let $i:\mathcal{C'}\hookrightarrow \mathcal{C}$ the inclusion functor of a full subcategory $\mathcal{C'}$ in $\mathcal{C}$. The relative cohomology with coefficient in a functor $F$ is defined by
$$\operatorname{H}^*(\mathcal{C},\mathcal{C'};F):=\operatorname{H}^*(\Hom(\frac{\mathcal{B}_{*}^{\mathcal{C}}}{\mathcal{B}_{*}^{\mathcal{C'}}},F))$$ and provides a long exact sequence 
$$\cdots \rightarrow \operatorname{H}^i(\mathcal{C};F) \rightarrow \operatorname{H}^i(\mathcal{C'};F\circ i) \rightarrow \operatorname{H}^{i+1}(\mathcal{C},\mathcal{C'};F) \rightarrow \operatorname{H}^{i+1}(\mathcal{C};F)\rightarrow \cdots.$$
\end{recalls}
\begin{notations}
Let $x$ and $y$ two objects of $\mathcal{C}$ (not necessary an ordered set), we say that $x\leq y$  if there exists a morphism from $x$ to $y$ in $\mathcal{C}$. We say that $x$ and $y$ are joinable if it exists an object $z$ such that $x\leq z \geq y$.
\begin{itemize} 
\item We denote $\widetilde{\mathcal{C}}_x$ the full subcategory of $\mathcal{C}$ whose objects are joinable to $x$.
\item We denote $\mathcal{C}_x$ the full subcategory of $\widetilde{\mathcal{C}}_x$ whose objects are $\ngeq x$.
\item We denote $\mathcal{C}_{\leq x}$ (resp. $\mathcal{C}_{<x}$ , $\mathcal{C}_{\geq x}$) the full subcategory of $\mathcal{C}$ whose objects are $\leq x$ (resp. $<x$ , $\geq x$).
\end{itemize}
\end{notations}

\begin{proposition}
Let $\mathcal{C}$ be a small category and $x\in \text{Ob}(\mathcal{C})$. Then 
$$\operatorname{H}^*(\mathcal{C};P_x^{\mathcal{C}})\cong \operatorname{H}^*(\widetilde{\mathcal{C}}_x;P_x^{\widetilde{\mathcal{C}}_x}).$$
\end{proposition}

\begin{proof}~~\\
There is an exact sequence of $\mathcal{C}$-modules $
 0\rightarrow \mathcal{B}_{*}^{\widetilde{\mathcal{C}}_x} \rightarrow \mathcal{B}_{*}^{\mathcal{C}} \rightarrow \mathcal{B}_{*}^{\mathcal{C}}/\mathcal{B}_{*}^{\widetilde{\mathcal{C}}_x} \rightarrow 0$. Since $\mathcal{B}_{*}^{\mathcal{C}}/\mathcal{B}_{*}^{\widetilde{\mathcal{C}}_x}$ is projective, by applying the functor $\Hom(-,P_x^{\mathcal{C}})$, we obtain the exact sequence
$$
 0\rightarrow \Hom(\mathcal{B}_{*}^{\mathcal{C}}/\mathcal{B}_{*}^{\widetilde{\mathcal{C}}_x},P_x^{\mathcal{C}}) \rightarrow \Hom(\mathcal{B}_{*}^{\mathcal{C}},P_x^{\mathcal{C}})\rightarrow \Hom(\mathcal{B}_{*}^{\widetilde{\mathcal{C}}_x},P_x^{\mathcal{C}})\rightarrow 0
.$$
For $i\in \mathbb{N}$, we have
$$\displaystyle \Hom_{}(\mathcal{B}_{i}^{\mathcal{C}}/\mathcal{B}_{i}^{\widetilde{\mathcal{C}}_x},P_x^{\mathcal{C}}) \cong \prod_{x_0\rightarrow \cdots \rightarrow x_i \notin \mathcal{N}_i\widetilde{\mathcal{C}}_x}P_x^{\mathcal{C}}(x_i)=0,$$ and therefore the isomorphism
$\Hom(\mathcal{B}_{*}^{\mathcal{C}};P_x^{\mathcal{C}}) \cong \Hom(\mathcal{B}_{*}^{\widetilde{\mathcal{C}}_x},P_x^{\mathcal{C}}).$
\end{proof}
\begin{proposition}\label{p1}
Let $\mathcal{C}$ be a small category and $x\in \text{Ob}(\mathcal{C})$. Then 
\begin{equation}\operatorname{H}^*(\mathcal{C};P_x^{\mathcal{C}})\cong\operatorname{H}^*(\widetilde{\mathcal{C}}_x,\mathcal{C}_x;P_x^{\widetilde{\mathcal{C}}_x}).\label{iso1}\end{equation}
\end{proposition}
\begin{proof}
The proof is analogous to the previous one. There is an exact sequence of $\mathcal{C}$-modules $0\rightarrow \mathcal{B}_{*}^{\mathcal{C}_x} \rightarrow \mathcal{B}_{*}^{\widetilde{\mathcal{C}}_x}\rightarrow \mathcal{B}_{*}^{\widetilde{\mathcal{C}}_x}/\mathcal{B}_{*}^{\mathcal{C}_x}\rightarrow 0.$
Apply the functor $\Hom(-,P_x^{\mathcal{C}})$, we obtain the exact sequence 
$$0\rightarrow \Hom(\mathcal{B}_{*}^{\widetilde{\mathcal{C}}_x}/\mathcal{B}_{*}^{\mathcal{C}_x},P_x^{\mathcal{C}}) \rightarrow \Hom(\mathcal{B}_{*}^{\widetilde{\mathcal{C}}_x},P_x^{\mathcal{C}})\rightarrow \Hom(\mathcal{B}_{*}^{\mathcal{C}_x},P_x^{\mathcal{C}})\rightarrow 0.$$
 for $i\in \mathbb{N}$, we have
$$\displaystyle \Hom_{}(\mathcal{B}_{i}^{\mathcal{C}_x},P_x^{\mathcal{C}}) \cong \prod_{x_0\rightarrow \cdots \rightarrow x_i \in \mathcal{N}_i\mathcal{C}_x}P_x^{\mathcal{C}}(x_i)=0,$$
 and therefore the isomorphism
$\Hom(\mathcal{B}_{*}^{\widetilde{\mathcal{C}}_x}/\mathcal{B}_{*}^{\mathcal{C}_x};P_x^{\mathcal{C}}) \cong \Hom(\mathcal{B}_{*}^{\widetilde{\mathcal{C}}_x},P_x^{\mathcal{C}}).$

\end{proof}

\bull{Partially ordered set.}
\medskip

In this paragraph, we give a condition for a simplicial poset to be a duality category.
\begin{recalls}[simplicial complexes]

A simplicial complex on a vertex set $V$ is a collection $\Delta$ of non-empty finite subsets of $V$ such that $\{v\}\in \Delta$ for all $v\in V$ and any non-empty subset of  $A\in \Delta$ is in $\Delta$. Call simplex an element of $\Delta$. The dimension of a simplex $F$ is defined by $\dim F=|F|-1$. The dimension of $\Delta$ is the maximum dimension of its elements.
\medskip

Let $A$ a simplex of a simplicial complex $\Delta$, the link of $A$ is the simplicial subcomplex of simplices $B$ joinable to $A$ and disjoint from $A$:
$$\operatorname{link}_A=\left\{B; A\cup B \in \Delta\text{ and } A\cap B=\emptyset \right\}.$$

A partially ordered set, shortened poset, is a category such that between two objects there is at most one morphism. To any simplicial complex $\Delta$, we  associate a poset $P(\Delta)$, whose objects are simplices of $\Delta$ and whose morphisms are inclusions. A poset $P$ is called simplicial if there exists a simplicial complex $\Delta$ such that $P\cong P(\Delta)$. 
The simplicial join of two simplicial complexes $\Delta$ and $\Delta'$ is the simplicial complex
$$\Delta \ast \Delta'=\left \{ A\cup B; A \in \Delta,~B\in \Delta'\right \}\cup \Delta \cup \Delta'.$$

Let $P=P(\Delta)$ and $Q=P(\Delta')$ be two simplicial posets. Call simplicial join the simplicial poset $P\ast Q:=P(\Delta\cup\Delta')$.
We denote $|\Delta|$ the geometric realization of a simplicial complex $\Delta$. The barycentric subdivision of a poset $P$ is a simplicial poset which we denote $\Delta(P)$. The geometric realization of $P$ is the one of $\Delta(P)$.
\end{recalls} 

For $\mathcal{C}$ a poset, the proposition \ref{p1} takes the following simple form:

\begin{proposition}\label{pair}
If $\mathcal{C}$ is a poset, then the isomorphism (\ref{iso1}) takes the simple form
\begin{equation}\operatorname{H}^*(\mathcal{C};P_x^{\mathcal{C}}) \cong \operatorname{H}^*(\widetilde{\mathcal{C}}_x,\mathcal{C}_x;\underline{\mathbb{Z}}).\label{eq3}\end{equation}
If, in addition, $\widetilde{\mathcal{C}}_x$ is contractible then
\begin{equation}\operatorname{H}^*(\mathcal{C},P_x^{\mathcal{C}}) \cong \widetilde{\operatorname{H}}^{*-1}(\mathcal{C}_x;\underline{\mathbb{Z}}).\label{eq4}\end{equation}
\end{proposition}

\begin{proof}
Let $\underline{\mathbb{Z}}^{\mathcal{C}_x}$ be the $\widetilde{\mathcal{C}}_x$-module, defined by $\underline{\mathbb{Z}}^{\mathcal{C}_x}(a)=\mathbb{Z}$ if $a\in \text{Ob}(\mathcal{C}_x)$ and $0$ otherwise. If $\mathcal{C}$ is a poset \textit{i.e.} there is at most one morphism between two objects, then we have the exact sequence of $\widetilde{\mathcal{C}}_x$-modules:
$$0\rightarrow \underline{\mathbb{Z}}^{\mathcal{C}_x} \rightarrow \underline{\mathbb{Z}}\rightarrow P_x^{\widetilde{\mathcal{C}}_x} \rightarrow 0.$$
Apply the functor $\Hom(\mathcal{B}_{*}^{\widetilde{\mathcal{C}}_x}/\mathcal{B}_{*}^{\mathcal{C}_x},-)$, we obtain the exact sequence
$$0\rightarrow \Hom(\mathcal{B}_{*}^{\widetilde{\mathcal{C}}_x}/\mathcal{B}_{*}^{\mathcal{C}_x},\underline{\mathbb{Z}}^{\mathcal{C}_x}) \rightarrow \Hom(\mathcal{B}_{*}^{\widetilde{\mathcal{C}}_x}/\mathcal{B}_{*}^{\mathcal{C}_x},\underline{\mathbb{Z}})\rightarrow \Hom(\mathcal{B}_{*}^{\widetilde{\mathcal{C}}_x}/\mathcal{B}_{*}^{\mathcal{C}_x},P_x^{\widetilde{\mathcal{C}}_x})\rightarrow 0.$$
 For $i\in \mathbb{N}$, we have
$$\displaystyle \Hom_{}(\mathcal{B}_{i}^{\widetilde{\mathcal{C}}_x}/\mathcal{B}_{i}^{\mathcal{C}_x},\underline{\mathbb{Z}}^{\mathcal{C}_x}) \cong \prod_{x_0\rightarrow \cdots \rightarrow x_i \notin \mathcal{N}_i\mathcal{C}_x}\underline{\mathbb{Z}}^{\mathcal{C}_x}(x_i)=0,$$
 and therefore the isomorphism
$$\Hom(\mathcal{B}_{*}^{\widetilde{\mathcal{C}}_x}/\mathcal{B}_{*}^{\mathcal{C}_x},\underline{\mathbb{Z}})\stackrel{\approx}{\rightarrow} \Hom(\mathcal{B}_{*}^{\widetilde{\mathcal{C}}_x}/\mathcal{B}_{*}^{\mathcal{C}_x},P_x^{\widetilde{\mathcal{C}}_x}).$$
The isomorphism (\ref{eq3}) follows by taking cohomology. The long exact sequence of reduced cohomology associated to the pair $(\widetilde{\mathcal{C}}_x,\mathcal{C}_x)$ gives the second isomorphism.

\end{proof}

\begin{example}
Let $\mathcal{P}$ the following poset 
$$
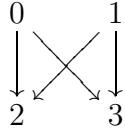
\xymatrix{
0 \ar[rd] \ar[d] &1 \ar[ld] \ar[d]\\
2  &3}$$
then
$$\widetilde{\mathcal{P}}_0=\widetilde{\mathcal{P}}_1=\mathcal{P} \text{    ,    } \mathcal{P}_0=1\text{   and   } \mathcal{P}_1=0.$$
We are in the situation where $ \mathcal{P}_0$ and $ \mathcal{P}_1$ are contractible, so
$$\operatorname{H}^1(\widetilde{\mathcal{P}_0},\mathcal{P}_0;\underline{\mathbb{Z}})=\operatorname{H}^1(\widetilde{\mathcal{P}_1},\mathcal{P}_1;\underline{\mathbb{Z}})=\mathbb{Z} \text{   and   } \operatorname{H}^0(\widetilde{\mathcal{P}_0},\mathcal{P}_0;\underline{\mathbb{Z}})=\operatorname{H}^0(\widetilde{\mathcal{P}_1},\mathcal{P}_1;\underline{\mathbb{Z}})=0.$$
Moreover, we have

$$\widetilde{\mathcal{P}}_2=\xymatrix{
0 \ar[d]  &1 \ar[ld] \\
2  &} ~~~~\text{      and      }~~~~\widetilde{\mathcal{P}}_3=\xymatrix{
0 \ar[rd]  &1 \ar[d]\\
  &3}$$ $\mathcal{P}_2$ and $\mathcal{P}_3$ consist of two points $0$ and $1$. We are in the situation where $\widetilde{\mathcal{P}}_2$ and $\widetilde{\mathcal{P}}_3$ are contractible. We find that 

$$\operatorname{H}^1(\widetilde{\mathcal{P}_2},\mathcal{P}_2;\underline{\mathbb{Z}})=\operatorname{H}^1(\widetilde{\mathcal{P}_3},\mathcal{P}_3;\underline{\mathbb{Z}})=\mathbb{Z} \text{    and    }
\operatorname{H}^0(\widetilde{\mathcal{P}_2},\mathcal{P}_2;\underline{\mathbb{Z}})=\operatorname{H}^0(\widetilde{\mathcal{P}_3},\mathcal{P}_3;\underline{\mathbb{Z}})=0.$$
The poset $\mathcal{P}$ is a duality category and the dualizing functor $D^1(\underline{\mathbb{Z}})$ is the constant module $\underline{\mathbb{Z}}$.
\end{example}

\begin{lemma}\label{join}
If $\mathcal{P}$ is a simplicial poset then
$$\widetilde{\mathcal{P}}_x \cong \mathcal{P}_{\leq x} \ast \operatorname{link}_x \text{   and   }\mathcal{P}_x \cong \mathcal{P}_{<x} \ast \operatorname{link}_x.$$
\end{lemma}
\begin{proof}
Let $\Delta$ be a simplicial complex such that $\mathcal{P} \cong P(\Delta)$. We can see via this isomorphism objects of $P$ as finite sets and morphisms as inclusions. Let $z$ an object of $\widetilde{\mathcal{P}}_x$, viewed as a finite set, is the disjoint union of $z\cap x$ and $z\backslash x$. The isomorphisms of the lemma are  constructed by associating to $z$ the object $(z\cap x)\cup (z\backslash x)$ of $\mathcal{P}_{\leq x} \ast \operatorname{link}_x$.
\end{proof}

\begin{theorem}\label{th1}
Let $\mathcal{P}$ a finite simplicial poset. If there exists an integer $n$ such that for all objects $x$ of $\mathcal{P}$
$$\widetilde{ \operatorname{H}}^{i}(\operatorname{link}_{x};\underline{\mathbb Z})=0 \text{   for   } i\neq n-\dim(x)-1$$ then $\mathcal{P}$ is a duality category.
\end{theorem}
\begin{proof}
The previous lemma shows that $\widetilde{\mathcal{P}}_x$ is contractible for every object $x$ of $\mathcal{P}$. We have then $$\operatorname{H}^*(\mathcal{P},P_x^{\mathcal{P}})=\widetilde{\operatorname{H}}^{*-1}(\mathcal{P}_{<x} \ast \operatorname{link}_x;\underline{\mathbb{Z}}).$$ Now, for a simplicial poset, $\mathcal{P}_{<x}$ is (co)homologically a sphere of dimension $\dim x-1$ so $\widetilde{\operatorname{H}}^{k}(\mathcal{P}_{<x} \ast \operatorname{link}_x;\underline{\mathbb{Z}})\cong  \widetilde{ \operatorname{H}}^{k-\dim(x)}(\operatorname{link}_{x};\underline{\mathbb Z})$ (see \protect{\cite[theorem 62.5]{Munkres}}). If $\widetilde{\operatorname{H}} ^i(\operatorname{link}_{x};\underline{\mathbb{Z}})=0 \text{   for   } i\neq n-\dim(x)-1$ then the reduced cohomology of $\mathcal{P}_x$ is concentrated in dimension $n-1$. It follows that $\operatorname{H}^i(\mathcal{P},P_x^{\mathcal{P}})=0$ for $i \neq n$. The assumption that $\mathcal{P}$ is finite ensures that $\underline{\mathbb{Z}}$ is of type $FP$ therefore the theorem \ref{dualite} allows to conclude.
\end{proof}

\medskip
The following key lemma due to Munkres, allows us to write a topological version of theorem~\ref{th1}.

\begin{lemma}[ ]\label{key}
Let $\Delta$ be a simplicial complex and $x$ a simplex of $\Delta$. Then
$$\operatorname{H}^i(|\Delta|,|\Delta|-x;\mathbb{Z}) \cong \widetilde{\operatorname{H}}^{i-dim(x)-1}(\operatorname{link}_x;\mathbb{Z}).$$

\end{lemma}
\begin{proof}
See \protect{\cite[lemma $3.3$ page $116$]{Munkres84}} .
\end{proof}

\begin{theorem}\label{th2}
Let $\Delta$ a finite simplicial complex of dimension $n$, then $P(\Delta)$ is a duality category if and only if the local cohomology of $|\Delta|$ at any point is concentrated in degree $n$. In this case:
 $$\operatorname{Ext}_{\modcat{P(\Delta)}}^i(\underline{\mathbb{Z}},-)\cong \operatorname{Tor}^{P(\Delta)}_{n-i}(D^n(\underline{\mathbb{Z}}),-)$$
where $D^n(\underline{\mathbb{Z}})$ is the functor defined on objects by $$D^n(\underline{\mathbb{Z}})(x)=\operatorname{H}^i(|\Delta|,|\Delta|-x;\mathbb{Z}).$$
\end{theorem}

\begin{remark}\label{general}
If $\mathcal{C}$ is a poset, then we have the isomorphisms 
\begin{equation}\operatorname{H}^*(\mathcal{C};P_x^{\mathcal{C}}) \cong \operatorname{H}^*(\mathcal{C},\mathcal{C}-\mathcal{C}_{\geq x};P_x^{\mathcal{C}}) \cong \operatorname{H}^*(\mathcal{C},\mathcal{C}-\mathcal{C}_{\geq x};\underline{\mathbb{Z}}).\end{equation}
The functor $P_x^{\mathcal{C}}$ is trivial on $\mathcal{C}-\mathcal{C}_{\geq x}$ so the third term in the long exact sequence associated to the pair $(\mathcal{C},\mathcal{C}-\mathcal{C}_{\geq x})$ is also trivial and this implies the first isomorphism. The second one follows from the long exact sequence of relative cohomology groups associated to the short exact sequence of $\mathcal{C}$-modules:
$$0\rightarrow \underline{\mathbb{Z}}^{\mathcal{C}-\mathcal{C}_{\geq x}} \rightarrow \underline{\mathbb{Z}}\rightarrow P_x^{\mathcal{C}} \rightarrow 0.$$ where $\underline{\mathbb{Z}}^{\mathcal{C}-\mathcal{C}_{\geq x}}$ is the $\mathcal{C}$-module, defined by $\underline{\mathbb{Z}}^{\mathcal{C}-\mathcal{C}_{\geq x}}(a)=\mathbb{Z}$ if $a\in \text{Ob}(\mathcal{C}-\mathcal{C}_{\geq x})$ and $0$ otherwise. This remark allows us to identify the local cohomology group $\operatorname{H}^i(|\Delta|,|\Delta|-x;\mathbb{Z})$ to $\operatorname{H}^i(\Delta,\Delta-\Delta_{\geq x};\mathbb{Z})$ for a simplicial complex $\Delta$.
\end{remark}

\medskip
\bull{Application: spherical buildings.}
\medskip 
\begin{definition}A spherical building is a simplicial complex which is the union of subcomplexes $\Sigma_i$ satisfying
 \begin{enumerate}
\item Each $\Sigma_i$ is a finite Coxeter complex.
\item For all simplices $A,B\in \Delta$, there is an apartment $\Sigma_i$ containing $A$ and $B$.
\item If $\Sigma_1$ and $\Sigma_2$ are two apartments containing $A$ and $B$, there exists an isomorphism $\Sigma_1\rightarrow \Sigma_2$ fixing $A$ and $B$.
\end{enumerate}
\end{definition}

\begin{theorem}[\protect{\cite[theorem 4.73 page 197]{Buildings}}]
If $\Delta$ is a spherical building of dimension $n$, then $|\Delta|$ has the homotopy type of a wedge of $n$-spheres.
\end{theorem}
\begin{proposition}[\protect{\cite[proposition 4.9 page 176]{Buildings}}]
If $\Delta$ is a spherical building, then $\operatorname{link}_A$ is a spherical building for any $A\in \Delta$.
\end{proposition}
For a spherical building of dimension $n$, $|\operatorname{link}_A|$ has the homotopy type of a wedge of spheres of dimension $n-\dim(A)-1$. We deduce from theorem~\ref{th1}:
\begin{corollary}\label{immeubledetits}
Let $\Delta$ a finite spherical building then the simplicial poset $P(\Delta)$ is a duality category.
\end{corollary}

\bull{Poincar\'e Duality for a manifold.}
\medskip

In this example, we find the classical Poincar\'e duality for a cohomological, compact and triangulated manifold. We recall that a cohomological manifold of dimension $n$ is a topological space $X$ such that for every $x\in X$, the local cohomology groups $\operatorname{H}^i(X,X-x)$ are trivial for $i\neq n$ and  $\operatorname{H}^n(X,X-x)=\mathbb{Z}$.

\begin{proposition}\label{constant}
Let $\Delta$ a finite triangulation of a cohomological manifold of dimension $n$.
The poset $P(\Delta)$ is a duality category. The $P(\Delta)^{op}$-module $D^n(\underline{\mathbb{Z}})$ is equal to $\mathbb{Z}$ for all object $x$ of $P(\Delta)$ and takes morphisms in $P(\Delta)$ to isomorphisms.
\end{proposition}
\begin{proof}
Since $\Delta$ is a cohomological manifold of dimension $n$, the local cohomology groups  $\operatorname{H}^i(|\Delta|,|\Delta|-x)$ are trivial for $i\neq n$ and equal to $\mathbb{Z}$ for $i=n$. Then ,by theorem \ref{th2},
the poset $P(\Delta)$ is a duality category and the $P(\Delta)^{op}$-module $D^n(\underline{\mathbb{Z}})$ is equal to $\mathbb{Z}$ on each object $x$ of $P(\Delta)$. Let $x,y$ two objects of $P(\Delta)$ such there is a morphism $\alpha:x\rightarrow y$ (this morphism correspond to the inclusion of $x$ in $y$ as simplices of $\Delta$). We show that $D^n(\underline{\mathbb{Z}})(\alpha)$ is an isomorphism. To reduce the amount of notation, we denote $\mathcal{C}$ the poset $P(\Delta)$. 
The functor $D^n(\underline{\mathbb{Z}})$ is equal to the abelian group $\mathbb{Z}$ on objects then, by the long exact sequence of cohomology groups associated to the short exact sequence $0\rightarrow P_y^{\mathcal{C}} \rightarrow P_x^{\mathcal{C}}\rightarrow \frac{P_x^{\mathcal{C}}}{P_y^{\mathcal{C}}} \rightarrow 0$ induced by $\alpha$, it's sufficient to show that $\operatorname{H}^n(\mathcal{C}; \frac{P_x^{\mathcal{C}}}{P_y^{\mathcal{C}}})$ is trivial. Let $z$ the object of $P(\Delta)$ that correspond to $y\backslash x$ in $\Delta$.\\
Claim 1: We have the isomorphism
\begin{equation}\label{cs}
\operatorname{H}^n(\mathcal{C}; \frac{P_x^{\mathcal{C}}}{P_y^{\mathcal{C}}})\cong \widetilde{\operatorname{H}}^{n-dim(x)-1}(\operatorname{link}_x-\operatorname{link}_{x_{\geq z}}; \mathbb{Z}).
\end{equation}
Claim 2: Denote $\operatorname{link}(t,\operatorname{link}_x)$ the $\operatorname{link}$ of a simplex $t$ in $\operatorname{link}_x$. Then $\operatorname{link}(t,\operatorname{link}_x)$ is equal to $\operatorname{link}_{x\cup t}$ in $\Delta$.\\
By lemma \ref{key}, the geometric realization $|\operatorname{link}_x|$ is a cohomology sphere hence the reduced cohomology of $\operatorname{link}_x-\operatorname{link}_{x_{\geq z}}$ is trivial since it's the reduced cohomology of $|\operatorname{link}_x|$ minus a point for which we refer to remark \ref{general} and lemma \ref{key}. We conclude that $\operatorname{H}^n(\mathcal{C}; \frac{P_x^{\mathcal{C}}}{P_y^{\mathcal{C}}})$ is equal to $0$.\\
Proof of the isomorphism (\ref{cs}) of claim 2: First we have that 
$$\operatorname{H}^n(\mathcal{C}; \frac{P_x^{\mathcal{C}}}{P_y^{\mathcal{C}}})\cong 
H^n(\widetilde{\mathcal{C}}_x-\widetilde{\mathcal{C}}_{x_{\geq z}};P_x^{\widetilde{\mathcal{C}}_x-\widetilde{\mathcal{C}}_{x_{\geq z}}} ).$$
This follows from the long exact sequence associated to the pair $(\mathcal{C},\widetilde{\mathcal{C}}_x-\widetilde{\mathcal{C}}_{x_{\geq z}})$ and the fact that
$$ \Hom_{}(\mathcal{B}_{i}^{\mathcal{C}}/\mathcal{B}_{i}^{\widetilde{\mathcal{C}}_x-\widetilde{\mathcal{C}}_{x_{\geq z}}},\frac{P_x^{\mathcal{C}}}{P_y^{\mathcal{C}}}) \cong \prod_{x_0\rightarrow \cdots \rightarrow x_i \notin \mathcal{N}_i\widetilde{\mathcal{C}}_x-\widetilde{\mathcal{C}}_{x_{\geq z}}}\frac{P_x^{\mathcal{C}}}{P_y^{\mathcal{C}}}(x_i)=0.$$
Secondly, similarly to the proof of lemma \ref{join}, we have
$$\widetilde{\mathcal{C}}_x-\widetilde{\mathcal{C}}_{x_{\geq z}} \cong \mathcal{C}_{\leq x}*(\operatorname{link}_x-\operatorname{link}_{x_{\geq z}}) \text{  and  }
\mathcal{C}_x-\mathcal{C}_{x_{\geq z}} \cong \mathcal{C}_{< x}*(\operatorname{link}_x-\operatorname{link}_{x_{\geq z}}).$$ 
Then by proposition \ref{pair} we have 
\begin{align*}
H^n(\widetilde{\mathcal{C}}_x-\widetilde{\mathcal{C}}_{x_{\geq z}};P_x^{\widetilde{\mathcal{C}}_x-\widetilde{\mathcal{C}}_{x_{\geq z}}})&\cong \widetilde{\operatorname{H}}^{n-1}(\mathcal{C}_x-\mathcal{C}_{x_{\geq z}}; \underline{\mathbb{Z}})\\
&\cong \widetilde{\operatorname{H}}^{n-1}(\mathcal{C}_{< x}*(\operatorname{link}_x-\operatorname{link}_{x_{\geq z}}); \underline{\mathbb{Z}})\\
&\cong \widetilde{\operatorname{H}}^{n-dim(x)-1}(\operatorname{link}_x-\operatorname{link}_{x_{\geq z}}; \underline{\mathbb{Z}}).
\end{align*}
Proof of the claim 2: First we show $\operatorname{link}(t,\operatorname{link}_x)\subset \operatorname{link}_{t\cup x}$. Let $v$ a simplex of $\operatorname{link}(t,\operatorname{link}_x)$. On the one hand, by definition of the $\operatorname{link}$ of simplicial complex, we have $v$ is disjoint from $t$ in $\operatorname{link}_x$ so $v$ is disjoint from $t\cup x$. On the other hand, $v$ is joinable to $t$ in $\operatorname{link}_x$ hence $v\cup t\in \operatorname{link}_x$ and therefore $v$ is joinable to $t\cup x$. It follows that $v\in \operatorname{link}_{t\cup x}$. Secondly, a simplex $v$ of $\operatorname{link}_{t\cup x}$ is disjoint from $t$ and joinable to $t$ in 
$\operatorname{link}_x$ then $v\in \operatorname{link}(t,\operatorname{link}_x)$.

\end{proof}

\begin{recalls}[\protect{\cite[theorem 3.26 page 236]{hatcher}}]
The manifold $|\Delta|$ is orientable if and only if $\operatorname{H}_n(|\Delta|,\mathbb{Z})\cong \mathbb Z$.
\end{recalls}

\begin{proposition}
The manifold $|\Delta|$ is orientable if and only if the dualizing functor $D^n(\underline{\mathbb{Z}})$ is the constant $P(\Delta)^{op}$-module $\underline{\mathbb{Z}}$.
\end{proposition}
\begin{proof}
Assume that $|\Delta|$ is orientable. The proposition~\ref{opdualite} gives 
$$\operatorname{Ext}_{\modcat{P(\Delta)^{op}}}^0(D(\underline{\mathbb Z}),\underline{\mathbb{Z}})\cong \operatorname{Tor}^{P(\Delta)^{op}}_{n}(\underline{\mathbb{Z}},\underline{\mathbb{Z}}).$$ The orientability of $|\Delta|$ is equivalent to $\operatorname{Tor}^{P(\Delta)^{op}}_{n}(\underline{\mathbb{Z}},\underline{\mathbb{Z}})\cong \mathbb Z$ hence $\Hom_{P(\Delta)^{op}\mathrm{-mod}}(D^n(\underline{\mathbb Z}), \underline{\mathbb Z})\cong \mathbb Z$. The previous proposition shows that $D^n(\underline{\mathbb Z})$ takes morphisms to isomorphisms and equal to $\mathbb Z$ on objects. If we fix a generator for $D^n(\underline{\mathbb Z})(x)=\mathbb Z$ for each object $x$ then the isomorphism $D^n(\underline{\mathbb Z})(y)=\mathbb{Z}\rightarrow D^n(\underline{\mathbb Z})(x)=\mathbb{Z}$ (whenever it exists a morphism between $x$ and $y$) is either $Id$ or $-Id$. The isomorphism $$\Hom_{P(\Delta)^{op}\mathrm{-mod}}(D^n(\underline{\mathbb Z}), \underline{\mathbb Z})\cong \mathbb Z$$ shows that we can choose a natural transformation $N:D^n(\underline{\mathbb Z})\rightarrow \underline{\mathbb Z}$ such that $N_x: D^n(\underline{\mathbb Z})(x)\rightarrow \underline{\mathbb Z}(x)$ is either $Id$ or $-Id$. This is not possible if $\Hom_{P(\Delta)^{op}\mathrm{-mod}}(D^n(\underline{\mathbb Z}), \underline{\mathbb Z})=0$ in the case where the manifold is not orientable. The natural transformation $N$ is an isomorphism and generates $\Hom_{P(\Delta)^{op}\mathrm{-mod}}(D^n(\underline{\mathbb Z}), \underline{\mathbb Z})$. Conversely if $D^n(\underline{\mathbb Z})\cong \underline{\mathbb Z}$ then the duality isomorphism $\operatorname{Ext}_{\modcat{P(\Delta)}}^i(\underline{\mathbb{Z}},-)\cong \operatorname{Tor}^{P(\Delta)}_{n-i}(D^n(\mathbb{Z}),-)$
implies $\operatorname{Tor}^{P(\Delta)}_{n}(\underline{\mathbb Z},\underline{\mathbb Z})\cong \mathbb Z$.
\end{proof}
\begin{proposition}
Let $X$ a cohomological, compact and triangulated manifold of dimension $n$. If $X$ is orientable then 
$$H_*(X,\mathbb{Z}) \cong H^{n-*}(X,\mathbb{Z}).$$
\end{proposition}
\begin{proof}
 The manifold $|\Delta|$ is orientable so the dualizing functor $D^n(\mathbb{Z})$ is the constant $P(\Delta)^{op}$- module $\underline{\mathbb{Z}}$. We obtain 
$$\operatorname{Ext}_{\modcat{P(\Delta)}}^i(\underline{\mathbb{Z}},\underline{\mathbb{Z}})\cong \operatorname{Tor}^{P(\Delta)}_{n-i}(\underline{\mathbb{Z}},\underline{\mathbb{Z}}).$$
 We thus find the classical Poincar\'e duality.
\end{proof}

\bull{Fundamental group of an aspherical manifold.}
\medskip
\begin{definition}
A manifold $X$ is said aspherical if the homotopy groups $\pi_k(X)$ are trivial for $k>1$. Equivalently, $X$ is an  Eilenberg-MacLane space of type $K(\pi_1(X);1)$.
\end{definition}

The aim of this section is to prove in our setting that the first fundamental group of an aspherical compact and triangulated manifold is a Poincar\'e duality group.

\subsubsection{Preliminaries}
 Let $\varphi:\mathcal{C} \rightarrow \mathcal{D}$ a functor between two small categories. The functor $\varphi$ induces a functor  $\varphi^{op}:\mathcal{C}^{op} \rightarrow \mathcal{D}^{op}$. Let $\text{Res}_{\varphi}:\mathcal{D}\mathrm{-mod} \longrightarrow \mathcal{C}\mathrm{-mod}$ (resp. $\text{Res}_{\varphi^{op}}:\mathcal{D}^{op}\mathrm{-mod} \longrightarrow  \mathcal{C}^{op}\mathrm{-mod}$) the functor precomposition by $\varphi$ (resp. $\varphi^{op}$). 
\begin{lemma}
The functors $\text{Res}_{\varphi}$ and $\text{Res}_{\varphi^{op}}$ have a left adjoints. We denote $\text{LK}_\varphi$ (resp. $\text{LK}_{\varphi^{op}}$) the left adjoint of $\text{Res}_\varphi$ (resp. $\text{Res}_{\varphi^{op}}$).
 \end{lemma} 
\begin{proof}
We make use of lemma \ref{dense} to construct $\text{LK}_\varphi$ as follows. Define $\text{LK}_\varphi$ on $\mathcal{G}$ by $\text{LK}_\varphi(P_a^{\mathcal{C}}\oplus P_a^{\mathcal{C}})=P_{\varphi(a)}^{\mathcal{D}}\oplus P_{\varphi(a)}^{\mathcal{D}}$ then for $F=\mathrm{colim}\  F_i$ with $F_i \in \text{Ob}\mathcal{G}$, define $\text{LK}_\varphi(F):=\mathrm{colim}\ \text{LK}_\varphi(F_i)$. The functor $\text{LK}_\varphi$ is a left adjoint to $\text{Res}_{\varphi}$.
Indeed for any object $G$ of $\mathcal{C}_T\mathrm{-mod}$ we have a natural isomorphism
 $$\Hom_{\modcat{\mathcal{D}}}(\text{LK}_\varphi(F),G)\cong \Hom_{\cmod}(F, \text{Res}_{\varphi}(G))$$ obtained as follows:
$$\begin{array}{cclcc}
\Hom_{\modcat{\mathcal{C}_T}}(\text{LK}_\varphi(\mathrm{colim}\ F_{i}), G) &\cong &\mathrm{lim}\  \Hom_{\modcat{\mathcal{C}_T}}(\text{LK}_\varphi(F_{i}),
G)&&\\
&\cong &\mathrm{lim}\ \Hom_{\cmod}(F_i, \text{Res}_{\varphi}(G)) &&\\
&\cong &\Hom_{\cmod}(\mathrm{colim}\ F_i, \text{Res}_{\varphi}(G))&&\\
&\cong
&\Hom_{\cmod}(F, \text{Res}_{\varphi}(G))
 .\end{array}$$
The construction of $\text{LK}_{\varphi^{op}}$ is similar.
\end{proof}
\begin{remark}
The notation $\text{LK}$ is an abbreviation of left Kan extension (see \protect{\cite[X.3]{categories}}).
\end{remark}
\begin{properties}\label{LK}\
\begin{enumerate}
\item $\text{LK}_\varphi(P_a^{\mathcal{C}})=P_{\varphi(a)}^{\mathcal{D}}$.
\item $F\otimes_{\mathcal{D}}\text{LK}_{\varphi}(G) \cong \text{Res}_{\varphi^{op}}{F}\otimes_{\mathcal{C}}G$ and $\text{LK}_{\varphi^{op}}(F)\otimes_{\mathcal{D}} G \cong F\otimes_{\mathcal{C}}\text{Res}_\varphi{G}$.

\item The functor $\text{LK}_{\varphi}$ is right exact. We denote $\mathbb{L}(\text{LK}_{\varphi})$ the left derived functor of $\text{LK}_{\varphi}$. To calculate  $\mathbb{L}(\text{LK}_{\varphi})$ of a functor $F$, we choose a projective resolution $\mathcal{P}_*$ of $F$ and we apply the functor $\text{LK}_{\varphi}$. We denote $\mathbb{L}_{i}\text{LK}_{\varphi}(F)$ the $n$'th homology group of $\mathbb{L}(\text{LK}_{\varphi})F=\mathbb{L}(\text{LK}_{\varphi})\mathcal{P}_*$. We have an isomorphism
 $$\mathbb{L}_{i}\text{LK}_{\varphi}(F)(d) \cong \operatorname{Tor}^{\cmod}_{i}(\text{Res}_{\varphi^{op}}P_d^{\mathcal{D}},F).$$

\item Let $\varphi|d$ denotes the fiber category over $d\in \text{Ob}\mathcal{D}$ whose objects are morphisms in $\mathcal{D}$ of the form $\varphi(a)\rightarrow d$. A morphism from $\varphi(a)\stackrel{\alpha}{\rightarrow} d$ to $\varphi(b)\stackrel{\beta}{\rightarrow} d$ is a morphism $h:a\rightarrow b$ in $\mathcal{C}$ such that $\beta \circ \varphi(h)=\alpha$. The fiber category $d|\varphi$ is the category $(\varphi^{op}|d)^{op}$. Let $j:\varphi|d \rightarrow \mathcal{C}$ the functor that assigns to a morphism $\varphi(a)\rightarrow d$ the object $a$ of $\mathcal{C}$ and is the identity on morphisms of $\varphi|d$. We keep the same notation for the analogue functors on $d|\varphi$ and $\varphi^{op}|d$. We have an isomorphism
$$\mathbb{L}_{i}\text{LK}_{\varphi}(F)(d)\cong \operatorname{Tor}^{\modcat{\varphi|d}}_{i}(\underline{\mathbb Z },\text{Res}_{j}F).$$
\end{enumerate}

\end{properties}
\begin{proof}\
\begin{enumerate}
\item This is a consequence of the isomorphisms
$$\Hom_{\modcat{\mathcal{D}}}(\text{LK}_{\varphi}P_a^{\mathcal{C}},G)\cong \Hom_{\modcat{\mathcal{C}}}(P_a^{\mathcal{C}},\text{Res}_{\varphi}G)\cong \Hom_{\modcat{\mathcal{D}}}(P_{\varphi(a)}^{\mathcal{D}},G).$$
\item Since the tensor products and $\text{LK}_{\varphi}$ commute with colimits, it's sufficient to prove the first statement for $G=P_a^{\mathcal{C}}$, which is a direct consequence of the first property and the isomorphisms (\ref{yonnedat}). The proof for the second statement is analogue.
\item The second property implies the isomorphisms
$$\mathbb{L}(\text{LK}_{\varphi}\mathcal{P}_*)(d)\cong P^d_{\mathcal{D}}\otimes_{\mathcal{D}}\text{LK}_{\varphi}\mathcal{P}_* \cong \text{Res}_{\varphi^{op}}{P^d_{\mathcal{D}}}\otimes_{\mathcal{C}}\mathcal{P}_*.$$
The statement follows by taking homology.
\item we have an isomorphism
$$\text{Res}_jP_x^{\mathcal{C}}\cong \bigoplus_{\varphi(x)\rightarrow d}{P_{\varphi(x)\rightarrow d}^{\varphi|d}}$$
given by the map associating to a morphism $$h\in \text{Res}_jP_x^{\mathcal{C}}(\varphi(a)\stackrel{f}{\rightarrow} d)=\mathbb{Z}[\Hom_{\mathcal{C}}(x,a)]$$ the same morphism $h$ as an element of $$P_{\varphi(x)\stackrel{f\circ \varphi(h)}{\rightarrow} d}^{\varphi|d}(\varphi(a)\stackrel{f}{\rightarrow} d).$$
This isomorphism and the isomorphisms (\ref{yonnedat}) imply
$$\text{LK}_{\varphi}(P_x^{\mathcal{C}})(d)\cong \underline{\mathbb{Z}}\otimes_{\varphi|d}\text{Res}_jP_x^{\mathcal{C}}.$$
Since $\text{LK}_{\varphi}$ and $\text{Res}_j$ commutes with colimits ($\text{Res}_j$ is an exact functor), we obtain $$\text{LK}_{\varphi}(F)(d)\cong \underline{\mathbb{Z}}\otimes_{\varphi|d}\text{Res}_jF$$
for any functor $F$ of $\cmod$. The functor $\text{Res}_j$ takes a standard projective to a direct sum of projectives and commutes with colimits then $\text{Res}_j$ preserves projectives. The statement follows by applying the last isomorphism to a projective resolution of $F$ and applying homology.  
\end{enumerate}
\end{proof}
We have $\text{LK}_\varphi(P_a^{\mathcal{C}})=P_{\varphi(a)}^{\mathcal{D}}$. It follows that $$\text{LK}_\varphi\mathcal{B}_{n}^{\mathcal{C}}\cong  \bigoplus_{x_0\rightarrow \cdots \rightarrow x_n \in \mathcal{N}_n\mathcal{C}}P^{\mathcal{D}}_{\varphi(x_n)}.$$
Let $d\in \text{Ob}(\mathcal{D})$. To give a chain $x_0\rightarrow \cdots \rightarrow x_n$ in $\mathcal{N}_n(\varphi |d)$ is equivalent to give a chain in $\mathcal{C}$ and a morphism in $\Hom_{\mathcal{D}}(\varphi(x_n),d)$. This correspondence induces an isomorphism 
$$\text{LK}_\varphi\mathcal{B}_{*}^{\mathcal{C}}:=\mathbb Z[\mathcal{N}_*(\varphi|-)].$$
Assume that the categories $\varphi |d$ are contractible for every $d\in \text{Ob}(\mathcal{D})$. Then the complex $\text{LK}_\varphi\mathcal{B}_{*}^{\mathcal{C}}$ is a projective resolution of the constant $\mathcal{D}$-module $\underline{\mathbb{Z}}$. The left Kan extension $\text{LK}_{\varphi}$ sends a resolution of type $FP$ of the constant $\mathcal{C}$-module $\underline{\mathbb{Z}}$ to a resolution of type $FP$ of the constant $\mathcal{D}$-module $\underline{\mathbb{Z}}$.

Let F be a $\mathcal{D}$-module. The adjunction isomorphism
$$\Hom_{\modcat{\mathcal{D}}}(\text{LK}_\varphi \mathcal{B}_{*}^{\mathcal{C}},F)\cong \Hom_{\cmod}(\mathcal{B}_{*}^{\mathcal{C}}, \text{Res}_{\varphi}F)$$
induces by taking cohomology the isomorphism
$$\operatorname{Ext}_{\modcat{\mathcal{D}}}^*(\underline{\mathbb{Z}},F)\cong \operatorname{Ext}_{\modcat{\mathcal{C}}}^{*}(\underline{\mathbb{Z}},\text{Res}_{\varphi}F).$$
For $F=P_d^{\mathcal{D}}$ we obtain the isomorphism $$\operatorname{Ext}_{\modcat{\mathcal{D}}}^*(\underline{\mathbb{Z}},P_d^{\mathcal{D}})\cong \operatorname{Ext}_{\modcat{\mathcal{C}}}^{*}(\underline{\mathbb{Z}},\text{Res}_{\varphi}P_d^{\mathcal{D}}).$$
Suppose $\mathcal{C}$ is a duality category then $$\operatorname{Ext}_{\modcat{\mathcal{D}}}^i(\underline{\mathbb{Z}},P_d^{\mathcal{D}})\cong \operatorname{Tor}^{\cmod}_{n-i}(D,\text{Res}_{\varphi}P_d^{\mathcal{D}}),$$
where $D$ is the dualizing $\mathcal{C}$-module of $\mathcal{C}$ and $n$ the projective dimension of $\underline{\mathbb{Z}}$.
As a consequence of properties \ref{LK}, we obtain the following isomorphisms 
\begin{align*}
\operatorname{Tor}^{\cmod}_{n-i}(D,\text{Res}_{\varphi}P_d)  &\cong \mathbb{L}_{n-i}\text{LK}_{\varphi^{op}}(D)(d)\\
&\cong \operatorname{Tor}^{\modcat{\varphi^{op}|d}}_{n-i}(\underline{\mathbb Z}, \text{Res}_jD)\\
&\cong \operatorname{Tor}^{\modcat{d|\varphi}}_{n-i}(\text{Res}_jD,\underline{\mathbb Z}).
\end{align*}

We summarize the above discussion in the following proposition
\begin{proposition}
Let $\varphi:\mathcal{C} \rightarrow \mathcal{D}$ a functor between two small categories. Assume that $\mathcal{C}$ is a duality category with dualizing functor $D$ with projective dimension $n$ and that the fibers $\varphi |-$ are contractible.
If $\operatorname{Tor}^{\modcat{d|\varphi}}_{n-i}(\text{Res}_jD,\underline{\mathbb Z})=0$ for $i\neq n$ then $\mathcal{D}$ is a duality category.
\end{proposition}
\subsubsection{Application}
Let  $\widetilde{\pi}_1(\mathcal{C})$ the groupoid obtained from $\mathcal{C}$ by inverting all morphisms.
The category $\widetilde{\pi}_1(\mathcal{C})$ is equivalent to the automorphism group of one of its objects which we denote  $\pi_1(\mathcal{C})$. 
This  group is isomorphic to the fundamental group of $|\mathcal{C}|$ (see \cite{Quillen}). 
Suppose that $|\mathcal{C}|$ is  $\text{K}(\pi_1(\mathcal{C});1)$ then the canonical functor  $\pi:\mathcal{C} \rightarrow  \pi_1(\mathcal{C})$ 
(we consider $\pi_1(\mathcal{C})$ as a category with one object $\bullet$) 
is an homotopy equivalence and the fibers $\pi|\bullet$ and $\bullet|\pi$ are contractible. 
Moreover, suppose that $\mathcal{C}$ is the simplicial poset associated to a triangulation of a cohomological compact orientable manifold then the dualizing functor $D$ is the constant $\mathcal{C}$-module $\underline{\mathbb{Z}}$ (proposition \ref{constant}). 
In this case, we have $\operatorname{Tor}^{\bullet|\pi}_{n-i}(D,\underline{\mathbb Z})=0$ for $i\neq n$ and equal to $\mathbb{Z}$ for $i=n$. 
It follows that $\pi_1(\mathcal{C})$ is a Poincar\'e duality group (a special case of \protect{\cite[example 1 page 222]{Brown}}).
%

%
\end{document}